\documentclass[11pt,a4paper]{amsart}
\usepackage{amssymb,amsmath,amsthm,enumerate,mathtools,float,pinlabel}
\usepackage{lscape,tabularx,makecell,multirow,stmaryrd}
\newtheorem{theorem}{Theorem}[section]
\newtheorem{prop}[theorem]{Proposition}
\newtheorem{lemma}[theorem]{Lemma}
\newtheorem{cor}[theorem]{Corollary}
\newtheorem{cor*}{Corollary}
\newtheorem{prop*}{Proposition}
\theoremstyle{definition}
\newtheorem{definition}[theorem]{Definition}
\newtheorem{defn}[theorem]{Definition}
\newtheorem{rem}[theorem]{Remark}

\newtheorem{example}[theorem]{Example}

\newcommand{\Mod}{\mathrm{Mod}}
\newcommand{\Aut}{\mathrm{Aut}_k}

\newcommand{\LMod}{\mathrm{LMod}}
\newcommand{\SMod}{\mathrm{SMod}}
\newcommand{\Homeo}{\mathrm{Homeo}}

\newcommand{\lcm}{\mathrm{lcm}}

\newcommand{\lb}{\llbracket}
\newcommand{\rb}{\rrbracket}
\renewcommand{\P}{\mathcal{P}}
\newcommand{\Z}{\mathbb{Z}}
\renewcommand{\O}{\mathcal{O}}
\newcommand{\F}{\mathcal{F}}
\newcommand{\G}{\mathcal{G}}
\newcommand{\D}{\mathcal{D}}

\newcommand{\orb}{\mathrm{orb}}

\makeatletter
\@namedef{subjclassname@2020}{\textup{2020} Mathematics Subject Classification}
\makeatother

\begin{document} 

\title[Split metacyclic actions on surfaces]{Split metacyclic actions on surfaces}

\author[N. K. Dhanwani]{Neeraj K. Dhanwani}
\address{Department of Mathematics\\
Indian Institute of Science Education and Research Mohali\\
Knowledge city, Sector 81, Manauli, PO, Sahibzada Ajit Singh Nagar\\
Mohali 140306, Punjab\\
India}
\email{neerajk.dhanwani@gmail.com}

\author[K. Rajeevsarathy]{Kashyap Rajeevsarathy}
\address{Department of Mathematics\\
Indian Institute of Science Education and Research Bhopal\\
Bhopal Bypass Road, Bhauri \\
Bhopal 462 066, Madhya Pradesh\\
India}
\email{kashyap@iiserb.ac.in}
\urladdr{https://home.iiserb.ac.in/$_{\widetilde{\phantom{n}}}$kashyap/}

\author[A. Sanghi]{Apeksha Sanghi}
\address{Department of Mathematics\\
Indian Institute of Science Education and Research Bhopal\\
Bhopal Bypass Road, Bhauri \\
Bhopal 462 066, Madhya Pradesh\\
India}
\email{apeksha16@iiserb.ac.in}



\maketitle

\begin{abstract}
Let $\text{Mod}(S_g)$ be the mapping class group of the closed orientable surface $S_g$ of genus $g\geq 2$. In this paper, we derive necessary and sufficient conditions under which two torsion elements in $\Mod(S_g)$ will have conjugates that generate a finite split non-abelian metacyclic subgroup of $\Mod(S_g)$. As applications of the main result, we give a complete characterization of the finite dihedral and the generalized quaternionic subgroups of $\Mod(S_g)$ up to a certain equivalence that we will call weak conjugacy. Furthermore, we show that any finite-order mapping class whose corresponding orbifold is a sphere, has a conjugate that lifts under certain finite-sheeted regular cyclic covers of $S_g$. Moreover, for $g \geq 5$, we show the existence of an infinite dihedral subgroup of $\Mod(S_g)$ that is generated by an involution and a root of a bounding pair map of degree $3$. Finally, we provide a complete classification of the weak conjugacy classes of the non-abelian finite split metacyclic subgroups of $\Mod(S_3)$ and $\Mod(S_5)$. We also describe nontrivial geometric realizations of some of these actions.
\end{abstract}

\section{Introduction}
\label{sec:intro}
Let $S_g$ be the closed orientable surface of genus $g \geq 0$, $\Homeo^+(S_g)$ be the group of orientation-preserving homeomorphisms on $S_g$, and let $\Mod(S_g)$ be the mapping class group of $S_g$. Given $F,G \in \Mod(S_g)$ of finite order, a pair of conjugates $F',G'$ (of $F,G$ resp.) may (or may not) generate a subgroup isomorphic to $\langle F, G \rangle$.  For example, consider the periodic mapping classes $F,G \in \Mod(S_7)$ represented by homeomorphisms $\F,\G \in \Homeo^+(S_g)$ (see~\cite{PKS} for details), as shown in the first subfigure of Figure~\ref{fig:motivation} below. 
\begin{figure}[htbp]
\centering
	\labellist
	\tiny
	\pinlabel $\F$ at 43 186
	\pinlabel $\G_3$ at 171 170
	\pinlabel $\pi$ at 182 152
    \pinlabel $\frac{2\pi}{4}$ at 53 170
	\pinlabel $\F_2$ at 300 186
	\pinlabel $\frac{2\pi}{4}$ at 312 170
	\pinlabel $\frac{2\pi}{4}$ at 300 145
	\pinlabel $\frac{6\pi}{4}$ at 300 95
	\pinlabel $\frac{2\pi}{4}$ at 300 60
	\pinlabel $\frac{6\pi}{4}$ at 300 11
	\pinlabel $\G$ at 102 83
	\pinlabel $\pi$ at 90 72
	\pinlabel $\G_1$ at 251 84
	\pinlabel $\pi$ at 235 72
	\pinlabel $\G_2$ at 369 82
	\pinlabel $\pi$ at 358 73
	\pinlabel $\F_1$ at 205 160
     \pinlabel $\frac{2\pi}{4}$ at 152 123
     \pinlabel $\frac{6 \pi}{4}$ at 196 9
     \pinlabel $\frac{6 \pi}{4}$ at 149 36
     \pinlabel $\frac{2 \pi}{4}$ at 197 150
	\endlabellist
	\includegraphics[width=55ex]{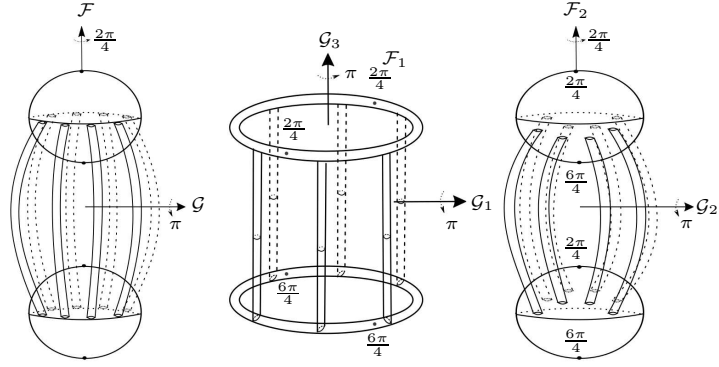}
	\caption{Split metacyclic subgroups of $\Mod(S_7)$ with conjugate generators.}
	\label{fig:motivation}
\end{figure}
\noindent From Figure~\ref{fig:motivation}, it is apparent that $\langle F, G \rangle \cong D_8$ (i.e. the dihedral group of order $8$). For $1 \leq i \leq 3$, we consider the conjugates $G_i$ of $G$, represented by the $\G_i \in \Homeo^+(S_7)$ and for $1 \leq j \leq 2$, we consider the conjugates $\F_j$ of $\F$ indicated in the (second and third) subfigures. In the second subfigure, we have marked the fixed points of a conjugate $\F_1$ of $\F$ (with the same local rotation angles as $\F$). Also, note that the third subfigure is different from the first (as an imbedding $S_7 \to \mathbb{R}^3$), since it has four pairs of tubes connecting the spheres, where in each pair, the tubes are aligned one behind the other. As it turns out, $\langle F_1 , G_1\rangle \cong \langle F_2 , G_2 \rangle \cong D_8$, but since $\F_1$  and $\G_3$ commute, we have $\langle F_1 , G_3 \rangle \cong \Z_4 \times \Z_2.$ Considering that the finite abelian subgroups of $\Mod(S_g)$ have been extensively studied~\cite{BW,DR,H1,M2}, this example motivates the following natural question:  Given $F',G' \in \Mod(S_g)$ of orders $n,m$ respectively, can one derive equivalent conditions under which there exist conjugates $F,G$ (of $F',G'$ resp.) such that $\langle F, G\rangle$ is a finite non-abelian \textit{split metacyclic subgroup of order $m\cdot n$ and twist factor $k$} admitting the presentation
$$\langle F,G \,|\, F^n = G^m = 1, G^{-1}FG = F^k \rangle  \cong \Z_n \rtimes_k \Z_m?$$ 
The main result in this paper answers this question in the affirmative for $k \neq 1$ (see Theorem~\ref{thm:main}). This result is a generalization of an analogous result from~\cite{DR} for two-generator finite abelian subgroups.


Given a  finite split (non-abelian) metacyclic subgroup $H = \langle F, G \rangle$ of $\Mod(S_g)$ as above, the Nielsen realization theorem~\cite{SK,JN} asserts that we may also view $H$ as a subgroup of $\Homeo^+(S_g)$ with an \textit{associated} $H$-action on $S_g$ inducing the branched cover $S_g \to S_g /H$. Given a branched cover $S_g \to S_g/\langle \F \rangle (= X)$ and a $\bar{G} \in \Mod(X)$ that lifts under this cover to a $G \in \Mod(S_g)$, it follows from Birman-Hilden theory~\cite{JB,BH1,BH2,BH3} that there is an exact sequence: 
\[1 \to \langle F \rangle \to \langle F, G \rangle \to \langle \bar{G} \rangle \to 1. \tag{$\dagger$} \]
\noindent A key ingredient in the proof of the main result is the derivation of elementary number-theoretic conditions under which such a $\bar{G}$ will have a conjugate that lifts so that the sequence ($\dagger$) splits (see Section~\ref{sec:main}). The proof integrates ideas from the theory of group actions on surfaces~\cite{SK,M1} with elements of Thurston's orbifold theory~\cite[Chapter 13]{WT}. In view of the Nielsen realization theorem, consider representatives $\F,\G \in \Homeo^+(S_g)$ of $F,G \in \Mod(S_g)$ (resp.) with the same orders. Another crucial aspect of the proof (of the main result) is the analysis of the geometric properties of the automorphism $\bar{\G}$ induced by $\G$ in $S_g/\langle \F \rangle$.

In Section~\ref{sec:appl}, we provide several applications of our main theorem. The first application concerns the finite dihedral subgroups of $\Mod(S_g)$. Let $D_{2n} = \Z_n \rtimes_{-1} \Z_2$ be the dihedral group of order $2n$. We derive the following characterization of dihedral subgroups of $\Mod(S_g)$ in Subsection~\ref{subsec:dihed}.
\begin{prop*}
Let $F \in \Mod(S_g)$ be of order $n$. Then there exists an involution $G \in \Mod(S_g)$ such that $\langle F,G \rangle \cong D_{2n}$ if and only if $F$ and $F^{-1}$ are conjugate in $\Mod(S_g)$. 
\end{prop*}

\noindent It is worth mentioning here that dihedral actions on Riemann surfaces have been classified in~\cite{BCGG}. 

For $n \geq 2$, the generalized quaternion group $Q_{2^{n+1}}$ is a metacyclic group of order $2^{n+1} $ that admits the presentation
$$\langle F,G \, | \, F^{2^n} = G^4 =1, \,  F^{2^{n-1}} = G^2, \, G^{-1}FG = F^{-1}\rangle.$$ 
\noindent In Subsection~\ref{subsec:quater}, we obtain the following characterization of generalized quaternionic actions on $S_g$ (see Proposition~\ref{prop:quat_action}). 
\begin{prop*}
For $g \geq 2$, $F \in \Mod(S_g)$ be of order $2^n$. Then there exists a $G \in \Mod(S_g)$ such that 
$\langle F,G \rangle \cong Q_{2^{n+1}}$ if and only if the $\langle \F, \G \rangle$-action on $S_g$ lifts to a $(\langle \tilde{\F} , \tilde{\G} \rangle \cong) \, \Z_{2^n} \rtimes_{-1} \Z_4$-action on $S_{2g-1}$ under the $2$-sheeted regular cyclic cover $S_{2g-1} \to S_g$ with deck  transformation group $\langle \tilde{\G}^2 \tilde{\F}^{2^{n-1}} \rangle$.
\end{prop*} 

For a periodic mapping class $F \in \Mod(S_g)$, the corresponding orbifold $\O_{\langle \F  \rangle} := S_g/\langle \F  \rangle \approx S_{g_0,r}$, where $S_{g_0,r}$ is the surface of genus $g_0 \geq 0$ with $r \geq 0$ marked points. It is known~\cite{G1} that $F$ is irreducible if and only if $\O_{\langle \F \rangle} \approx S_{0,3}$. In Subsection~\ref{subsec:liftable}, we provide a characterization of the split metacyclic subgroups $\langle F,G \rangle$ of $\Mod(S_g)$ when $F$ is irreducible (see Corollary~\ref{cor:irred_F}). Let $\LMod_p(S_g)$ (resp. $\SMod_p(S_{n(g-1)+1})$) be the liftable (resp. symmetric) mapping class groups of a finite $n$-sheeted regular cyclic cover $p : S_{n(g-1)+1} \rightarrow S_g$. In this context, we have the following result. 

\begin{prop*}
For $g,n \geq 2$, let $p : S_{n(g-1)+1} \rightarrow S_g$ be a regular cover with deck transformation group $\Z_n= \langle \F \rangle$. Then any involution $G' \in \Mod(S_g)$ has a conjugate $G \in \LMod_p(S_g)$ with a lift $\tilde{G} \in \SMod_p(S_{n(g-1)+1})$ such that $\langle F, \tilde{G} \rangle \cong D_{2n}$.
\end{prop*}

\noindent Moreover, we provide  sufficient conditions for the liftability of a periodic mapping class (under $p$) whose corresponding orbifold is a sphere (see Propositions~\ref{prop:lift_sph_ordern_cpt} - \ref{prop:lift_sph_no_ordern_cpt}). As a consequence, we obtain the following corollary.
\begin{cor*}
For $g\geq 2$ and prime $n$, let $p : S_{n(g-1)+1} \rightarrow S_g$ be a regular $n$-sheeted cover with deck transformation group $\langle \F \rangle \cong \Z_n$. Let $G' \in \Mod(S_g)$ be of order $m$ such that the genus of $\O_{\langle \G' \rangle}$ is zero. Then $G'$ has a conjugate $G \in \LMod_{p}(S_g)$ with a lift  $\tilde{G} \in \SMod_p(S_{n(g-1)+1})$ such that $\langle F, \tilde{G} \rangle \cong \Z_n \rtimes_k \Z_m$ if there exists $k \in \Z_n^{\times}$ such that $|k| =m$.
\end{cor*}

Consider an \textit{infinite metacyclic group}~\cite{CEH} that admits a presentation of the form 
$$\langle x,y \,|\, y^{2m} = 1, y^{-1}xy = x^{-1} \rangle.$$ 
When $m = 1,$ we call such a group an \textit{infinite dihedral group.} By a \textit{root of a mapping class $F \in\Mod(S_g)$ of degree $n$}, we mean a $G \in \Mod(S_g)$ such that $G^n = F$. In Subsection~\ref{subsec:inf_meta}, we use the theory developed in~\cite{KR2,KP}, to construct roots of multitwists (i.e. products of powers of commuting Dehn twists) in $\Mod(S_g)$ which together with certain mapping classes of order $2m$ generate infinite split metacyclic subgroups of $\Mod(S_g)$ (of the form described above) for $g \geq 5$ (see Proposition~\ref{prop:inf_split_meta}). In particular, for $m = 1$, we have the following corollary. 
\begin{cor*}
For $g \geq 5$, there exists an infinite dihedral subgroup of $\Mod(S_g)$ that is generated by an involution and a root of a bounding pair map of degree $3$.
\end{cor*}
 In Section~\ref{sec:hyp_str}, we classify the finite non-abelian split metacyclic subgroups of $\Mod(S_3)$ and $\Mod(S_5)$ up to a certain weaker notion of conjugacy that we call \textit{weak conjugacy} (see Definition~\ref{defn:weak_conj}), which arises naturally in our setting. It may be noted that similar classifications for $2 \leq g \leq 4$ can also be obtained through the techniques developed in~\cite{OVB,AB,HK}. Finally, we apply the results in~\cite{PKS} to provide an algorithm for determining the hyperbolic structures that realize split metacyclic subgroups as groups of isometries. We conclude the paper by giving nontrivial geometric realizations of some finite split metacyclic subgroups of $\Mod(S_3)$ and $\Mod(S_5)$.

\section{Preliminaries}
\label{sec:prelims}
 
\subsection{Fuchsian groups} Let $\Homeo^+(S_g)$ denote the group of orientation-preserving homeomorphisms on $S_g$, and let $H < \Homeo^+(S_g)$ be a finite group. A faithful and properly discontinuous $H$-action on $S_g$ induces a branched covering $$S_g \to \mathcal{O}_H := S_g/H$$ with $\ell$ cone points $x_1,\ldots ,x_{\ell}$ on the quotient orbifold $\mathcal{O}_H \approx S_{g_0}$ (which we will call the \textit{corresponding orbifold}) of orders $n_1, \ldots ,n_{\ell}$, respectively. Then the orbifold fundamental group $\pi_1^{\orb}(\O_H)$ of $\O_H$ has a presentation given by
\begin{equation}
\label{eqn:orb-pres}
\left\langle \alpha_1,\beta_1,\dots,\alpha_{g_0},\beta_{g_0}, \xi_1,\dots,\xi_{\ell} \, |\, \xi_1^{n_1},\dots,\xi_\ell^{n_{\ell}},\,\prod_{j=1}^{\ell} \xi_j \prod_{i=1}^{g_0}[\alpha_i,\beta_i]\right\rangle.
\end{equation}
In classical parlance, $\pi_1^{\orb}(\O_H)$ is also known as a \textit{Fuchsian group}~\cite{SK1, M1} with signature
$$\Gamma(\O_H) := (g_0;n_1,\ldots,n_{\ell}),$$ and the relation $\prod_{j=1}^{\ell} \xi_j \prod_{i=1}^{g_0}[\alpha_i,\beta_i]$ appearing in its presentation is called the \textit{long relation}.
\noindent From Thurston's orbifold theory~\cite[Chapter 13]{WT}, we obtain exact sequence
\begin{equation}
\label{eq:surf_kern}
1 \rightarrow \pi_1(S_g) \rightarrow \pi_1^{\orb}(\O_H) \xrightarrow{\phi_H}  H \rightarrow 1.
\end{equation}
\noindent In this context, we will require the following result due to Harvey~\cite{H1}.
 
\begin{lemma}
\label{Harvey Condition} 
A finite group $H$ acts faithfully on $S_g$ with $\Gamma(\O_H) = (g_0;n_1,\dots,n_{\ell})$ if and only if it satisfies the following two conditions: 
\begin{enumerate}[(i)]
\item $\displaystyle \frac{2g-2}{|H|}=2g_0-2+\sum_{i=1}^{\ell}\left(1-\frac{1}{n_i}\right)$, and 
\item  there exists a surjective homomorphism $\phi_H:\pi_1^{\orb}(\O_H) \to H$ that preserves the orders of all torsion elements of $\pi_1^{\orb}(\O_H)$.
\end{enumerate}
\end{lemma}
 
\subsection{Cyclic actions on surfaces } For $g \geq 1$, let $F \in \Mod(S_g)$ be of order $n$. The Nielsen-Kerckhoff theorem~\cite{SK,JN} asserts that $F$ is represented by a \textit{standard representative} $\F \in \Homeo^+(S_g)$ of the same order. We refer to both $\F$ and the group it generates, interchangeably, as a \textit{$\Z_n$-action on $S_g$}.  Each cone point $x_i \in \O_{\langle \F \rangle}$ lifts to an orbit of size $n/n_i$ on $S_g$, and the local rotation induced by $\F$ around the points in each orbit is given by $2 \pi c_i^{-1}/n_i$, where $\gcd(c_i,n_i)=1$ and $c_i c_i^{-1} \equiv 1 \pmod{n_i}$. Further, it is known (see~\cite{H1} and the references therein) that the exact sequence in \ref{eq:surf_kern} takes the following form 
\begin{equation*}
\label{eq:surf_kern}
1 \rightarrow \pi_1(S_g) \rightarrow \pi_1^{\orb}(\O_{\langle \F \rangle}) \xrightarrow{\phi_{\langle \F \rangle}}  \langle \F \rangle \rightarrow 1, 
\end{equation*}
\noindent where $\phi_{\langle \F \rangle} (\xi _i) = \F^{(n/n_i)c_i}$, for $1 \leq i \leq \ell$. We will now introduce a tuple of integers that encodes the conjugacy class of a $\Z_n$-action on $S_g$.  

\begin{definition}\label{defn:data_set}
A \textit{data set of degree $n$} is a tuple
$$
D = (n,g_0, r; (c_1,n_1),\ldots, (c_{\ell},n_{\ell})),
$$
where $n\geq 2$, $g_0 \geq 0$, and $0 \leq r \leq n-1$ are integers, and each $c_i \in \Z_{n_i}^\times$ such that:
\begin{enumerate}[(i)]
\item $r > 0$ if and only if $\ell = 0$ and $\gcd(r,n) = 1$, whenever $r >0$,
\item each $n_i\mid n$,
\item $\lcm(n_1,\ldots ,\widehat{n_i}, \ldots,n_{\ell}) = N$, for $1 \leq i \leq \ell$, where $N = n$, if $g_0 = 0$,  and
\item $\displaystyle \sum_{j=1}^{\ell} \frac{n}{n_j}c_j \equiv 0\pmod{n}$.
\end{enumerate}
The number $g$ determined by the Riemann-Hurwitz equation
\begin{equation}\label{eqn:riemann_hurwitz}
\frac{2-2g}{n} = 2-2g_0 + \sum_{j=1}^{\ell} \left(\frac{1}{n_j} - 1 \right) 
\end{equation}
is called the {genus} of the data set, denoted by $g(D)$.
\end{definition}

\noindent Note that quantity $r$ (in Definition~\ref{defn:data_set}) will be non-zero if and only if $D$ represents a free rotation of $S_g$ by $2\pi r/n$, in which case, $D$ will take the form $(n,g_0,r;)$. We will not include $r$ in the notation of a data set, whenever $r = 0$.  

By the Nielsen-Kerckhoff theorem, the canonical projection $\Homeo^+(S_g)\to \Mod(S_g)$ induces a bijective correspondence between the conjugacy classes of finite-order maps in $\Homeo^+(S_g)$ and the conjugacy classes of finite-order mapping classes in $\Mod(S_g)$. This leads us to the following lemma (that follows from~\cite[Theorem 3.8]{KP} and~\cite{H1}), which allows us to use data sets to describe the conjugacy classes of cyclic actions on $S_g$. 

\begin{lemma}\label{prop:ds-action}
For $g \geq 1$ and $n \geq 2$, data sets of degree $n$ and genus $g$ correspond to conjugacy classes of $\Z_n$-actions on $S_g$. 
\end{lemma}

\noindent We will denote the data set corresponding to the conjugacy class of a periodic mapping class $F$ by $D_F$. For compactness of notation, we also write a data set $D$ (as in Definition~\ref{defn:data_set}) as
$$D = (n,g_0,r; ((d_1,m_1),\alpha_1),\ldots,((d_{\ell'},m_{\ell'}),\alpha_{\ell'})),$$
where $(d_i,m_i)$ are the distinct pairs in the multiset $S = \{(c_1,n_1),\ldots,(c_{\ell},n_{\ell})\}$, and the $\alpha_i$ denote the multiplicity of the pair $(d_i,m_i)$ in the multiset $S = \{(c_1,n_1),\ldots,(c_{\ell},n_{\ell})\}$. Further, we note that every cone point $[x] \in \mathcal{O}_{\langle \F \rangle}$ corresponds to a unique pair in the multiset $S$ appearing in $D_F$, which we denote by $\P_x := (c_{x},n_{x})$. 

Given $u \in \Z_m^{\times}$ and $\G \in H \leq \Homeo^+(S_g)$ be of order $m$, let $\mathbb{F}_{\G} (u,m)$ denote the set of fixed points of $\G$ with induced rotation angle $2\pi u^{-1}/m$. Let $C_H(\G)$ be the centralizer of $\G \in H$ and $\sim$ denote the conjugation relation between any two elements in $H$. We conclude this subsection by stating the following result from the theory of Riemann surfaces~\cite{TB}, which we will use in the proof of our main theorem.

\begin{lemma}
\label{lem:riem_surf_fix}
Let $H < \Homeo^+(S_g)$ of finite order with $\Gamma(\O_H)  = (g_0; n_1,\ldots,n_{\ell})$, and let $\G \in H$ be of order $m$. Then for $u \in \Z_m ^{\times}$, we have
\[\displaystyle |\mathbb{F}_{\G} (u,m)| = |C_H(\G)| \cdot \sum_{\displaystyle \substack{1 \leq i \leq \ell \\ m \mid n_i \\ \G \sim \phi_H(\xi_i)^{n_i u/m}}} \frac{1}{n_i}.\]
\end{lemma}

\subsection{Hyperbolic structures realizing cyclic actions}
\label{subsec:hyp_str}
 Given a finite subgroup $H < \Mod(S_g)$, let $\text{Fix}(H)$ denote the subspace of fixed points in the Teich\"{m}uller space $\text{Teich}(S_g)$ under the action of $H$. When $H$ is cyclic, a method for constructing the hyperbolic metrics representing the points in $\text{Fix}(H)$ was described in~\cite{BPR} and~\cite{PKS}, thereby yielding explicit solutions to the Nielsen realization problem~\cite{SK,JN}. This method involved the construction of an arbitrary periodic element in $\Mod(S_g)$ (that is not realizable as a rotation of $S_g$) by the ``compatibilities" of irreducible Type 1 components, which are uniquely realized as rotations of certain special hyperbolic polygons with side-pairings. 

A mapping class that is not reducible is called \textit{irreducible}. Let $F \in \Mod(S_g)$ be of order $n$. Gilman~\cite{G1} showed that $F$ is irreducible if and only if $\Gamma(\O_{\langle \F\rangle})$ has the form $(0;n_1,n_2,n_3)$ (i.e. the quotient orbifold $\O_{\langle \F\rangle}$ is a sphere with three cone points.) Following the nomenclature in~\cite{BPR,PKS}, $F$ is \textit{rotational} if $\F$ is either of order $2$, or $\F$ has at most $2$ fixed points. A non-rotational $F$ is said to be of \textit{Type 1} if $\Gamma(\O_{\langle \F\rangle})=(g_0;n_1,n_2,n)$, otherwise, it is called a \textit{Type 2} action. The following result describes the unique hyperbolic structure that realizes an irreducible Type 1 action.

\begin{theorem}\label{res:1}
For $g \geq 2$, consider a irreducible Type 1 action $F \in {\Mod}(S_g)$ with $$D_F = (n,0; (c_1,n_1),\linebreak (c_2,n_2), (c_3,n)).$$ Then $F$ can be realized explicitly as the rotation $\displaystyle \theta_F = \frac{2 \pi c_3^{-1}}{n}$ of a hyperbolic polygon $\P_F$ with a suitable side-pairing $W(\P_F)$, where $\P_F$ is a hyperbolic  $k(F)$-gon with
$$ k(F) := \begin{cases}
2n, & \text { if } n_1,n_2 \neq 2, \text{ and } \\
n, & \text{otherwise, }
\end{cases}$$
and for $0 \leq m\leq n-1$, 
$$ 
W(\P_F) =
\begin{cases}
\displaystyle  
  \prod_{i=1}^{n} a_{2i-1} a_{2i} \text{ with } a_{2m+1}^{-1}\sim a_{2z}, & \text{if } k(F) = 2n, \text{ and } \\
\displaystyle
 \prod_{i=1}^{n} a_{i} \text{ with } a_{m+1}^{-1}\sim a_{z}, & \text{otherwise,}
\end{cases}$$
where $\displaystyle z \equiv m+qj \pmod{n}$ with $q= (n/n_2)c_3^{-1}$ and $j=n_{2}-c_{2}$.
\end{theorem} 

Further, it was shown~\cite{PKS} that the process of realizing an arbitrary non-rotational action $F$ of order $n$ using these unique hyperbolic structures realizing irreducible Type 1 components involved two broad types of processes. 
\begin{enumerate}[(a)]
\item \textit{$k$-compatibility.} In this process, for $i =1,2$, we take a pair of irreducible Type 1 mapping classes $F_i \in \Mod(S_{g_i})$ such that the $\langle \F_i \rangle$-action on $S_{g_i}$ induces a pairs of \textit{compatible orbits} of size $k$ (where the induced local rotation angles add upto 0 modulo $2\pi$). We remove (cyclically permuted) $\langle \F_i \rangle$-invariant disks around points in the compatible orbits and then identify the resulting boundary components realizing a periodic mapping class $F \in \Mod(S_{g_1+g_2+k-1})$. An analogous construction can also be performed using a pair of orbits induced by a single $\langle \F' \rangle$-action on $S_g$ to realize a periodic mapping class $F \in \Mod(S_{g+k})$.
\item \textit{Permutation additions and deletions.} The \textit{addition of a permutation component} involves the removal of (cyclically permuted) invariant disks around points in an orbit of size $n$ induced by an $\langle \F \rangle$-action on $S_g$ and then pasting $n$ copies of $S_{g'}^1$ (i.e. $S_{g'}$ with one boundary component) to the resultant boundary components. This realizes a action on $S_{g+ng'}$ with the same fixed point and orbit data as $F$. The reversal of this process is called a \textit{permutation deletion. }
 \end{enumerate}
 
\noindent Thus, in summary, we have the following: 
\begin{theorem}\cite[Theorem 2.24]{PKS}\label{res:2}
For $g \geq 2$, a non-rotational periodic mapping class in $\Mod(S_g)$ can be realized through finitely many $k$-compatibilities, permutation additions, and permutation deletions on the unique structures of type $\P_F$ realizing irreducible Type 1 mapping classes. 
\end{theorem}

\noindent A final, but yet vital ingredient in the realization of split metacyclic actions is the following elementary lemma, which is a direct generalization of \cite[Lemma 6.1]{DR}. 

\begin{lemma}
\label{lem:fix_metacyclic}
Let $H  = \langle F, G \rangle$ be a finite metacyclic subgroup of $\Mod(S_g)$. Then $$\mathrm{Fix}(H)=\mathrm{Fix}(\langle F\rangle)\cap\mathrm{Fix}(\langle G\rangle).$$
\end{lemma}

\subsection{Split metacyclic actions on surfaces} Given integers $m,n \geq 2$, and $k \in \Z_n^{\times}$ such that $k^m \equiv 1 \pmod{n}$, a \textit{finite split metacyclic action of order $mn$} (written as $m \cdot n$) on $S_g$ is a tuple $(H, (\G,\F))$, where 
$H < \Homeo^+(S_g)$, and $$H = \langle \F,\G \,|\, \F^n = \G^m = 1, \G^{-1}\F\G = \F^k \rangle.$$ We will call the multiplicative class $k$ the \textit{twist factor} of the split metacyclic action $(H, (\G,\F))$. As we are only interested in  non-abelian split metacyclic subgroups, we will assume from here on that $k \neq 1$. Note that in classical notation $H \cong \Z_n \rtimes_k \Z_m$. As $\langle \F \rangle \lhd H$, it is known~\cite{TB,TWT} that $\G$ would induce a $\bar{\G} \in \Homeo^+(\O_{\langle \F \rangle})$ that preserves the set of cone points in $\O_{\langle \F \rangle}$ along with their orders. We will call $\bar{\G}$, \textit{the induced automorphism on $\O_{\langle \F \rangle}$ by $\G$,} and we formalize this notion in the following definition. 

\begin{definition}
\label{defn:ind_auto}
Let $H < \Homeo^+(S_g)$  be a finite cyclic group with $|H| = n$. We say an $\bar{\F} \in \Homeo^+(\O_H)$  is an \textit{automorphism of $\O_H$} if for $[x],[y] \in \O_H$, $k \in \Z_n^{\times}$ and $\bar{\F}([x]) = [y]$, we have: 
\begin{enumerate}[(i)]
\item $n_x = n_y$, and 
\item $k c_x = c_y$. 
\end{enumerate} 
We denote the group of automorphisms of $\O_H$ by $\Aut(\O_H)$.
\end{definition}
\noindent  We note that the concept of an induced orbifold automorphism in Definition~\ref{defn:ind_auto} is more general than the one that was used in the abelian case~(\cite{DR}), which required a more rigid condition that $c_x = c_y$. The following lemma, which provides some basic properties of the induced map $\bar{\G}$, is a split metacyclic analog of~\cite[Lemma 3.1]{DR}. 

\begin{lemma}
\label{lem:ind_aut}
Let $\G,\F \in {\Homeo}^+(S_g)$ be maps of orders $m,n$, respectively, such that $\G^{-1} \F \G = \F^k$, and let $H = \langle \F\rangle$. Then:
\begin{enumerate}[(i)]
\item  $\G$ induces a $\bar{\G}\in \Aut(\O_H)$ such that $$\O_H/\langle \bar{\G}\rangle = S_g/\langle \F,\G\rangle,$$ 
\item $|\bar{\G}| \text{ divides } |\G|$, and 
\item $|\bar{\G}| < m$ if and only if $\F^{l}=\G^u$, for some $0< l<n$ and $0< u< m$.
\end{enumerate}
\end{lemma}

We will now formalize the notion of weak conjugacy from Section~\ref{sec:intro}. 

\begin{defn}
\label{defn:weak_conj}
Two finite split metacyclic actions $(H_1, (\G_1, \F_1))$ and $(H_2, (\G_2, \F_2))$ of order $m \cdot n$ and twist factor $k$ are said to be \textit{weakly conjugate} if there exists an isomorphism, $\psi: \pi_1^{\orb}(\O_{H_1}) \cong \pi_1^{\orb}(\O_{H_2})$ and an isomorphism $\chi : H_1 \to H_2$ such that the following conditions hold. 
\begin{enumerate}[(i)]
\item $\chi((\G_1,\F_1)) = (\G_2,\F_2)$.
\item For $i = 1,2$, let $\phi_i: \pi_1^{\orb}(\O_{H_i}) \to H_i$ be the surface kernel (appearing in the exact sequence~(\ref{eq:surf_kern}) in Section~\ref{sec:prelims}). Then $(\chi \circ \phi_{H_1})(g) = (\phi_{H_2} \circ \psi)(g), \text{ whenever } g \in \pi_1^{\orb}(\O_{H_1})$ is of finite order.
\item The pair $(\G_1, \F_1)$ is conjugate (component-wise) to the pair $(\G_2, \F_2)$ in $\Homeo^+(S_g).$
\end{enumerate}
\noindent The notion of weak conjugacy defines an equivalence relation on split metacyclic actions on $S_g$ and the equivalence classes thus obtained will be called \textit{weak conjugacy classes}. 
\end{defn}

\begin{rem}
\label{rem:ext_weak_conj}
By virtue of the Nielsen-Kerckhoff theorem, the notion of weak conjugacy in Definition~\ref{defn:weak_conj} naturally extends to an analogous notion in $\Mod(S_g)$ via the natural association 
$$(\langle \F, \G \rangle,(\G,\F)) \leftrightarrow  (\langle F, G \rangle,(G,F)).$$ 
\end{rem}


\noindent For simplicity, we will now introduce the following notation. 
\begin{defn}
Let $F,G \in \Mod(S_g)$ be a finite order map of orders $n,m$ respectively. Then for some $k \in \Z_n^\times \setminus \{1 \}$, we say (in symbols) that $\lb F,G \rb_k = 1$ if there exists conjugates $F',G'$ (of $F,G$ resp.) such that $(G')^{-1} F' G' = (F')^{k}$.
\end{defn}

\noindent We conclude this subsection with the following  crucial remark. 

\begin{rem}
\label{rem:eq_class_part}
Let $H <\Mod(S_g)$ be a finite split metacyclic subgroup, and let $I(H)$ denote the isomorphism class of $H$ (in $\Mod(S_g)$). By Remark~\ref{rem:ext_weak_conj}, we have 
\begin{gather*}
I(H) = \{H' : H' \cong H \text{ and } (H',(G',F')) \text{ represents a weak conjugacy class} \\  \text{for some } F',G' \in H' \text{ such that } H' = \langle F', G' \rangle\}.
\end{gather*}
\noindent Consequently, periodic mapping classes $F,G \in \Mod(S_g)$ satisfy $\lb F,G \rb_k = 1$ if and only if there exists conjugates $F',G'$ (of $F,G$ resp.) such that the triple $(\langle F',G' \rangle, (G',F'))$ represents a weak conjugacy class associated with a finite split metacyclic subgroup (of twist factor $k$) of $\Mod(S_g)$.
\end{rem}

\section{Main theorem}
\label{sec:main}
In this section, we establish the main result of the paper by deriving equivalent conditions under which torsion elements $F,G \in \Mod(S_g)$ would satisfy $\lb F, G \rb_k = 1$. We will introduce an abstract tuple of integers that will capture each weak conjugacy class associated with a finite split metacyclic subgroup of $\Mod(S_g)$.
\begin{defn}
\label{defn:meta_cyc_dataset}
	A \textit{split metacyclic data set of degree $m \cdot n$, twist factor $k$, and genus $g \geq 2$} is a tuple
$$((m \cdot n,k),g_0;[(c_{11},n_{11}),(c_{12},n_{12}),n_1], \cdots , [(c_{\ell 1},n_{\ell 1}),(c_{\ell 2},n_{\ell 2}),n_{\ell} ]),$$
	where $m,n \geq 2$, the $n_{ij}$ are positive integers for $1 \leq i \leq \ell, ~ 1 \leq j \leq 2$, and $k \in \Z_n^{\times}$ such that $k^m \equiv 1 \pmod{n}$, satisfying the following conditions.
	
		\begin{enumerate}[(i)]
		\item $\displaystyle	\frac{2g-2}{mn} = 2g_0 - 2 + \sum_{i=1}^{\ell} \left( 1 - \frac{1}{n_i} \right) .$
		\item \begin{enumerate}
			\item For each $i,j$, $n_{i1} \mid m$, $n_{i2} \mid n$, either $\text{gcd}(c_{ij},n_{ij}) = 1$ or $c_{ij}=0$, and $c_{ij}=0$ if and only if $n_{ij} = 1$.
			\item For each $i$, $n_i = n_{i1} \cdot \beta_i$, where $\beta_i$ is least positive integer such that $$\displaystyle c_{i2} \frac{n}{n_{i2}} \left( \sum_{i' = 0}^{n_{i1} \beta_i -1} k^{c_{i1} \frac{m}{n_{i1}} i'} \right) \equiv 0 \pmod{n}.$$
			 
		\end{enumerate}
	\item  $\displaystyle \sum_{i = 1}^{\ell} c_{i1} \frac{m}{n_{i1}} \equiv 0 \pmod{m}$. 
	\item Defining  $\displaystyle A := \sum_{i=1}^{\ell} c_{i2} \frac{n}{n_{i2}} \prod_{s = i+1}^{\ell} k^{c_{s1} \frac{m}{n_{s1}}}$ and $d:= \gcd(n,k-1)$, we have
	$$A \equiv \begin{cases}
	0 \pmod{n}, & \text{if } g_0=0, \text{ and} \\
	d\theta \pmod{n}, \text{ for } \theta \in \Z_n, & \text{if } g_0 \geq 1.
	\end{cases}$$
	\item If  $g_0 = 0$, there exists $(p_1, \ldots , p_{\ell v}), (q_1, \ldots, q_{\ell v}) \in \Z^{\ell v}$ and $a,v \in \mathbb{N}$ such that the following conditions hold. 
	\begin{enumerate}
	\item $ \displaystyle \sum_{i'=1}^{\ell v} p_{i'} c_{i1} \frac{m}{n_{i1}} \equiv 1 \pmod{m}$ and $$\sum_{i'=1}^{\ell v} c_{i2} \frac{n}{n_{i2}} \left(\sum_{s=1}^{p_{i'}} k^{c_{i1} \frac{m}{n_{i1}}(p_{i'} -s)}\right) \left(\prod_{t' = i'+1}^{\ell v} k^{c_{t1} \frac{m}{n_{t1}}}\right) \equiv 0 \pmod{n}.$$
	\item $\displaystyle \sum_{i'=1}^{\ell v} q_{i'} c_{i1} \frac{m}{n_{i1}} \equiv 0 \pmod{m}$ and 
	$$\sum_{i'=1}^{\ell v} c_{i2} \frac{n}{n_{i2}} \left(\sum_{s=1}^{q_{i'}} k^{c_{i1} \frac{m}{n_{i1}}(q_{i'} -s)}\right) \left(\prod_{t' = i'+1}^{\ell v} k^{c_{t1} \frac{m}{n_{t1}}}\right) \equiv 1 \pmod{n}, \text{ where } $$ 
$$i \equiv
	\begin{cases} 
	i' \pmod{\ell}, & \text{if } i' \neq a \ell, \\
	\ell & \text{otherwise,}\\
	 \end{cases} 
\,\,t \equiv
	\begin{cases} 
	t' \pmod{\ell}, & \text{if } t' \neq a \ell, \text{ and}\\
	\ell, & \text{otherwise.}\\
	 \end{cases}$$
	\end{enumerate}
	\item If  $g_0 = 1$, there exists $(p_1, \ldots , p_{\ell v}), (q_1, \ldots, q_{\ell v}) \in \Z^{\ell v}$ and $m',n' \in \Z, ~ a,v \in \mathbb{N}$ such that $m' \mid m$ and $n' \mid n$, satisfying the following conditions.
	\begin{enumerate}
	\item $ \displaystyle \sum_{i'=1}^{\ell v} p_{i'} c_{i1} \frac{m}{n_{i1}} \equiv m' \pmod{m}$ and $$\sum_{i'=1}^{\ell v} c_{i2} \frac{n}{n_{i2}} \left(\sum_{s=1}^{p_{i'}} k^{c_{i1} \frac{m}{n_{i1}}(p_{i'} -s)}\right) \left(\prod_{t' = i'+1}^{\ell v} k^{c_{t1} \frac{m}{n_{t1}}}\right) \equiv 0 \pmod{n}.$$
	\item $\displaystyle \sum_{i'=1}^{\ell v} q_{i'} c_{i1} \frac{m}{n_{i1}} \equiv 0 \pmod{m},$ 
	$$\sum_{i'=1}^{\ell v} c_{i2} \frac{n}{n_{i2}} \left(\sum_{s=1}^{q_{i'}} k^{c_{i1} \frac{m}{n_{i1}}(q_{i'} -s)}\right) \left(\prod_{t' = i'+1}^{\ell v} k^{c_{t1} \frac{m}{n_{t1}}}\right) \equiv n' \pmod{n},$$ 
	where $$i \equiv 
	\begin{cases} 
	i' \pmod{\ell}, & \text{if } i' \neq a \ell, \\
	\ell, & \text{otherwise,}
	 \end{cases}
	\,\, t \equiv \begin{cases} 
	t' \pmod{\ell} & \text{if } t' \neq a \ell, \text{ and}\\
	\ell & \text{otherwise.}
	 \end{cases}$$
	\item $A \equiv -\beta k^{\alpha} + \beta \pmod{n}$, where $$\displaystyle \lcm \left( \frac{m}{m'},\frac{m}{\gcd(m,\alpha)} \right) = m \text{ and } \lcm \left( \frac{n}{n'},\frac{n}{\gcd(n,\beta)} \right) = n.$$ Furthermore, we set $\alpha = 1$, when $m'=0$, and $\beta =1$, when $n'=0$.
	\end{enumerate}
	\end{enumerate}	 		
\end{defn}

\noindent We will now show that the split metacyclic data sets of genus $g$ are in one-to-one correspondence with the weak conjugacy classes of split metacyclic subgroups of $\Mod(S_g)$.

\begin{prop}
\label{prop:main}
For integers $n,m,g \geq 2$, the split metacyclic data sets of degree $m \cdot n$ with twist factor $k$ and genus $g$ correspond to the weak conjugacy classes of $\Z_n \rtimes_k \Z_m$-actions on $S_g$.
\end{prop}

\begin{proof}
Let $\D$ be a split metacyclic data set of degree $m \cdot n$ with twist factor $k$ and genus $g$ (as in Definition~\ref{defn:meta_cyc_dataset} above). We need to show that $\D$ corresponds to the weak conjugacy class of a $\Z_n \rtimes_k \Z_m$-action on $S_g$ represented by $(H,(\G,\F))$. To this effect, we first establish the existence of an epimorphism $\phi_H : \pi_1^{orb} (\O_H) \rightarrow H$ which preserves the order of torsion elements. Let the presentations of $H$ and $\Gamma(\O_H)$ be given by
\begin{gather*} 
H \cong \Z_n \rtimes_k \Z_m = \langle \F,\G\, | \,\F^n = \G^m = 1, \G^{-1}\F\G = \F^{k} \rangle \text{ and} \\
\langle \alpha_1, \beta_1 , \cdots , \alpha_{g_0} , \beta_{g_0}, \xi_1 , \cdots , \xi_{\ell} | \xi_1^{n_1} = \cdots = \xi_{\ell}^{n_\ell} = \prod_{j=1}^{\ell} \xi_j \prod_{i=1}^{g_0}[\alpha_i , \beta_i] = 1 \rangle, 
\end{gather*} respectively.
	We consider the map
	$$\displaystyle \xi_i \xmapsto{\phi_H}  \G^{c_{i1} \frac{m}{n_{i1}}} \F^{c_{i2} \frac{n}{n_{i2}}}, \text{ for } 1 \leq i \leq \ell.$$
	As $|\G^{c_{i1} \frac{m}{n_{i1}}}| = n_{i1}$ and $|\F^{c_{i2} \frac{n}{n_{i2}}}| = n_{i2},$ condition (ii) of Definition~\ref{defn:meta_cyc_dataset} would imply that $\phi_H$ is a map which preserves the order of torsion elements. 
For clarity, we break the argument for the surjectivity of $\phi_H$ into three cases. 

First, we consider the case when $g_0 = 0$. Conditions (iii) and (iv) show that $\phi_H$ satisfies the long relation $\prod_{i=1}^{\ell} \xi_i = 1$ and the surjectivity of $\phi_H$ follows from condition (v). 
	
	When $g_0 \geq 2$, $\pi_1^{orb} (\O_H)$ has additional hyperbolic generators (viewing them as isometries of the hyperbolic plane), namely the $\alpha_i$ and the $\beta_i$. Extending $\phi_H$ by mapping $\alpha_1 \xmapsto{\phi_H} \G, \beta_1 \xmapsto{\phi_H} \F$ yields an epimorphism. Moreover, by carefully choosing the $\alpha_i$ and the $\beta_i$, for $i \geq 2$, conditions (iii) and (iv) would together ensure that the long relation $\prod_{j=1}^{\ell} \xi_j \prod_{i=1}^{g_0}[\alpha_i ,\beta_i] = 1$ is satisfied. 
	
	When $g_0 = 1$, $\pi_1^{orb} (\O_H)$ has two additional hyperbolic generators, namely the $\alpha_1$ and the $\beta_1$. We extend $\phi_H$ by defining $\alpha_1 \xmapsto{\phi_H} \G^{\alpha}$ and $\beta_1 \xmapsto{\phi_H} \F^{\beta}$, and apply conditions (iv) and (vi) to obtain the desired epimorphism.
	
	It remains to show that $\D$ determines $\F,\G \in \Homeo^{+}(S_g)$ up to conjugacy (i.e. condition (iii) of Definition~\ref{defn:weak_conj}). Let $D_{\bar{G}} = (m,g_0;(c_{11},n_{11}), \hdots, (c_{\ell 1},n_{\ell 1}))$
represent the conjugacy class of the action $\bar{\G}$ induced on the orbifold $\O_{\langle \F \rangle}$ by the action $\G \in \Homeo^+(S_g)$. We note that by Lemma~\ref{lem:ind_aut}, $\Gamma(\O_{\langle \F \rangle})$ has the form
$$(g_1;\underbrace{\frac{n_1}{n_{11}}, \hdots \frac{n_1}{n_{11}}}_{\frac{m}{n_{11}} \text{ times}}, \hdots, \underbrace{\frac{n_{\ell}}{n_{\ell 1}}, \hdots, \frac{n_\ell}{n_{\ell1}}}_{\frac{m}{n_{\ell1}} \text{ times}}),$$
where if $n_i/n_{i1} = 1$, for some $1 \leq i \leq \ell$, then we exclude it from the signature, and $g_1 = g(D_{\bar{G}})$ is determined by Equation~(\ref{eqn:riemann_hurwitz}) of Definition~\ref{defn:data_set}. So, we get
\begin{gather*}
\displaystyle D_F = (n,g_1; (d_{11},\frac{n_1}{n_{11}}), \ldots, (d_{1\frac{m}{n_{11}}},\frac{n_1}{n_{11}}) , \hdots ,  \\
(d_{\ell 1},\frac{n_\ell}{n_{\ell1}}), \ldots, (d_{\ell \frac{m}{n_{\ell 1}}},\frac{n_\ell}{n_{\ell 1}}))
\end{gather*}
where $$d_{i1}n_{i1} \equiv c_{i2}\frac{n_i}{n_{i2}} \sum_{j'=1}^{n_{i1}}k^{c_{i1}\frac{m}{n_{i1}}(j' -1)} \pmod{n_i},$$
$$d_{ij_i} \equiv d_{i1} k^{\gamma(j_i -1)} \pmod{\frac{n_i}{n_{i1}}} ~ 1 \leq i \leq l, ~ 1 \leq j_i \leq \frac{m}{n_{i1}},$$
and $\gamma$ is the least positive integer such that $|k| \mid \gamma \frac{m}{n_{i1}}$. Moreover, by applying Lemma~\ref{lem:riem_surf_fix}, we see that 
\begin{gather*}\displaystyle D_G = (m,g_2;((u_{ij},m_i),\frac{m_i |\mathbb{f}_{\G^{\frac{m}{m_i}}}(u_{ij},m_i)|}{m}) : u_{ij} \in \mathbb{Z}_{m_i}^{\times}, \, m_i \mid m, \\  \mathrm{ and} \,\gcd(u_{ij},m_i) = 1),
\end{gather*}
where $$\displaystyle |\mathbb{f}_{\G^{\frac{m}{m_i}}}(u_{ij},m_i)|= |\mathbb{F}_{\G^{\frac{m}{m_i}}}(u_{ij},m_i)| - \sum_{\substack{m_{i'} \in \mathbb{N} \\ m_{i'} \neq m_i \\ m_i | m_{i'} | m}} \sum_{\substack{ (u_{i'j'},m_{i'}) = 1 \\  u_{ij} \equiv u_{i'j'} (\text{mod} \, m_i)}} |\mathbb{f}_{\G^{\frac{m}{m_{i'}}}}(u_{i'j'},m_{i'})|$$ and $g_2$ is determined by Equation~(\ref{eqn:riemann_hurwitz}) of Definition~\ref{defn:data_set}.
	
	Conversely, consider the weak conjugacy class of $\Z_n \rtimes_k \Z_m$-actions on $S_g$ represented by $(H, (\G,\F))$, where $H = \langle \F,\G \rangle$. So, Lemma~\ref{Harvey Condition} would imply that there exists a surjective homomorphism $$\phi_H : \pi_1^{orb}(\O_H) \to H : \displaystyle \xi_i \xmapsto{\phi_H} \G^{c_{i1} \frac{m}{n_{i1}}} \F^{c_{i2} \frac{n}{n_{i2}}}, \text{ for } 1 \leq i \leq \ell, $$
which is order-preserving on the torsion elements. This yields a split metacyclic data set of degree $m \cdot n$ with twist factor $k$ and genus $g$ as in Definition~\ref{defn:meta_cyc_dataset}. By Lemma~\ref{Harvey Condition}, this tuple satisfies condition (i) of Definition~\ref{defn:meta_cyc_dataset}, while condition (ii) follows from the fact that $\phi_H$ is order-preserving on torsion elements. Conditions (iii)-(iv) follow from the long relation satisfied by $\pi_1^{orb}(\O_H)$, and condition (v)-(vi) are implied by the surjectivity of $\phi_H$. Thus, we obtain the split metacyclic data set of degree $m \cdot n$ with twist factor $k$ and genus $g$, and the result follows.
\end{proof}

\noindent We denote the data sets $D_F$ and $D_G$ (representing the cyclic factors of $H$) derived from the split metacyclic data set $\D$ appearing in the proof of Proposition~\ref{prop:main} by $\D_1$ and $\D_2$, respectively. Thus, our main theorem will now follow from Remark~\ref{rem:eq_class_part} and Proposition~\ref{prop:main}.

\begin{theorem}[Main theorem]
\label{thm:main}
Let $F,G \in \Mod(S_g)$ be of orders $n,m$, respectively. Then $\lb F,G \rb_k = 1$ if and only if there exists a split metacyclic data set $\D$ of degree $m \cdot n$, twist factor $k$, and genus $g$ such that $\D_1 = D_F$ and $\D_2 = D_G$.
\end{theorem}


\noindent We conclude this section with an example of a split metacyclic action of order $16$ on $S_5$. 

\begin{example}
\label{exm:Z4_Z4_S3}
The split metacyclic data set $\D = ((4 \cdot 4,-1),1;[(0,1),(1,2),2])$ encodes the weak conjugacy class of a $\Z_4 \rtimes_{-1} \Z_4$-action on $S_5$ represented by $(\langle \F, \G\rangle, (\G,\F))$, where $$D_F = (4,1;(1,2),(1,2),(1,2),(1,2)) \text{ and }D_G = (4,2,1;).$$
The geometric realization of this action is illustrated in Figure~\ref{fig:Z4Z4_S5} below. 
 \begin{figure}[H]
	\centering
	\labellist
	\tiny
	\pinlabel $(1,2)$ at 115 75
	\pinlabel $(1,2)$ at 115 11
	\pinlabel $(1,2)$ at 115 140
	\pinlabel $(1,2)$ at 115 225
	\pinlabel $\G$ at 218 195
	\pinlabel $\frac{\pi}{2}$ at 210 177
	\pinlabel $(1,2)$ at 88 105
	\pinlabel $(1,2)$ at -15 105
	\pinlabel $(1,2)$ at 140 100
	\pinlabel $(1,2)$ at 240 100
	\pinlabel $(1,4)$ at 180 200
	\pinlabel $(1,4)$ at 50 205
	\pinlabel $(1,4)$ at 180 10
	\pinlabel $(1,4)$ at 50 10
	\pinlabel $(3,4)$ at 5 55
	\pinlabel $(3,4)$ at 220 50
	\pinlabel $(3,4)$ at 5 155
	\pinlabel $(3,4)$ at 220 150
	\endlabellist
	\includegraphics[width=33ex]{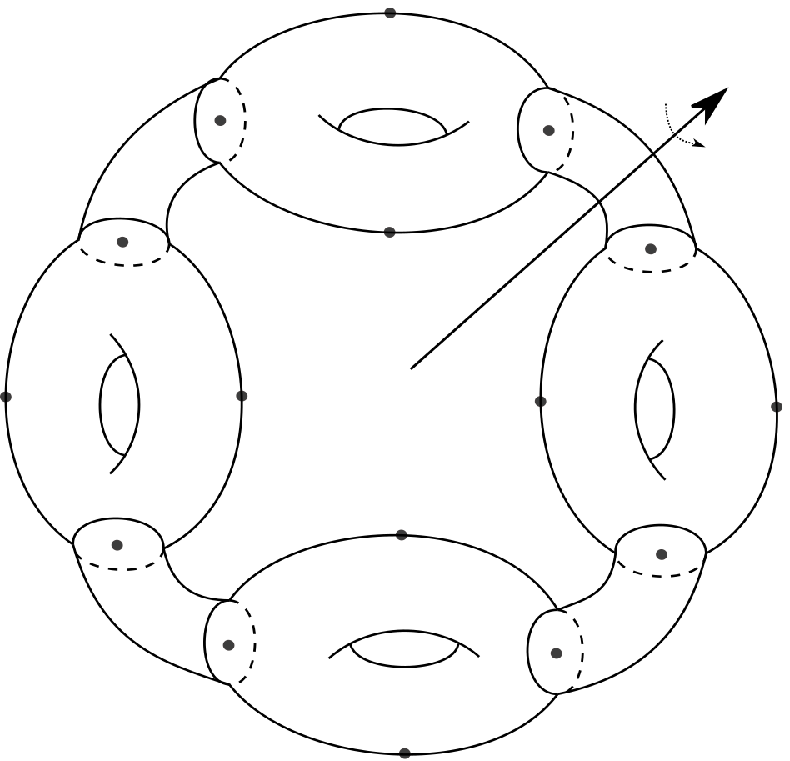}
	\caption{Realization of a $\mathbb{Z}_4 \rtimes_{-1} \mathbb{Z}_4$-action in $\Mod(S_5)$.}
	\label{fig:Z4Z4_S5}
	\end{figure}
\end{example}
\noindent Note that the pairs of integers appearing in Figure~\ref{fig:Z4Z4_S5} represent the compatible orbits involved in the realization of $\F$. Here, the action $\F$ is realized via two 1-compatibilities  between the action $\F'$ on two copies of $S_2$ with $D_{F'} = (4,0;((1,2),2),(1,4),(3,4))$. Furthermore, the action $\F'$ is realized by a 1-compatibility between the action $\F''$ on two copies of $S_1$ with $D_{F''} = (4,0;(1,2),(1,4),(3,4))$.
\section{Applications}

\label{sec:appl}

\subsection{Dihedral groups}
\label{subsec:dihed}
Let $D_{2n} = \Z_n \rtimes_{-1} \Z_2$ be the dihedral group of order $2n$. We will call a split metacyclic data set of degree $2\cdot n$ and twist factor $-1$ a \textit{dihedral data set}. A simple computation reveals that a dihedral data set $$((2 \cdot n,-1),g_0;[(c_{11},n_{11}),(c_{12},n_{12}),n_1], \cdots , [(c_{\ell 1},n_{\ell 1}),(c_{\ell 2},n_{\ell 2}),n_{\ell} ]),$$ would have the property that $(c_{j1},n_{j1}) \in \{(0,1),(1,2)\}$, for $1 \leq j \leq \ell$. The following is an immediate consequence of Proposition~\ref{prop:main}. 

\begin{cor}
For $g \geq 2$ and $n \geq 3$, dihedral data sets of degree $2\cdot n$ and genus $g$ correspond to the weak conjugacy classes of $D_{2n}$-actions on $S_g$.
\end{cor}

\noindent The following proposition provides an alternative characterization of a $D_{2n}$-action in terms of the generator of its factor subgroup of order $n$.

\begin{prop}
	Let $F \in \Mod(S_g)$ be of order $n$. Then there exists an involution $G \in \Mod(S_g)$ such that $\langle F, G \rangle \cong D_{2n}$ if and only if $D_F$ has the form 
	\[(n,g_0,r; ((c_1,n_1),(-c_1,n_1), \ldots, (c_{s},n_{s}),(-c_{s},n_{s})). \tag{*}\] 
\end{prop}

\begin{proof}
Suppose that $D_F$ has the form $(*)$. Then $\O_{\langle \F \rangle}$ is an orbifold of genus $g_0$ with $2s$ cone points $[x_1],[y_1],\ldots, [x_s],[y_s]$, where $\P_{x_i} = (c_i,n_i)$ and 
  $\P_{y_i} = (-c_i,n_i)$, for $1 \leq i \leq s$. Up to conjugacy, let $\bar{\G} \in \Aut( \O_{\langle \F \rangle})$ be the hyperelliptic involution so that $\bar{\G}([x_i])= [y_i]$, for $1 \leq i \leq s$. To prove our assertion, it would suffice to show the existence of an involution $\G \in \Homeo^+(S_g)$ that induces $\bar{\G}$. This amounts to show that there exists a split metacyclic data set $\D$ of degree $2 \cdot n$ with twist factor $-1$ encoding the weak conjugacy class $(H,(\G,\F))$ so that $D_G$ has degree $2$. Consider the tuple
\begin{gather*} 
\D = ((2 \cdot n, -1),0; \underbrace{[(1,2),(0,1),2], \hdots , [(1,2),(0,1),2]}_{t-2 \text{ times}}, [(1,2),(c_{(t-1)2},n_{(t-1)2}),2] \\ [(1,2),(c_{t2},n_{t2}),2],  [(0,1),(c_{1},n_{1}),n_1], \hdots , [(0,1),(c_{s},n_{s}),n_s]), 
\end{gather*}
where $t= 2g_0+2,$ $$\small ((c_{(t-1)2},n_{(t-1)2}),  (c_{t2},n_{t2})) = \begin{cases}
((0,1), (1, -\sum_{i=1}^{s} c_i \frac{n}{n_i} (\text{mod}\,n))), & \text{if } g_0 = 0, \text{ and} \\
((1,n) ,(1,1-\sum_{i=1}^{s} c_i \frac{n}{n_i} (\text{mod} \, n))), & \text{if } g_0 > 0.
\end{cases}$$
\noindent It follows immediately that $\D$ satisfies conditions (i)-(iv) of Definition~\ref{defn:meta_cyc_dataset}.  As $t \geq 2$, by taking $v =1$, we may choose $(p_1, \hdots , p_{t+s}) = (1,0, \hdots , 0)$ to conclude that $\D$ also satisfies condition (v)(a). Since $t = 2 \iff g_0 = 0$, and when $g_0 = 0$, we have that $\text{lcm}(n_1, \hdots , n_s) = n$, from which condition (v)(b) follows. Finally, for the case when $g_0 \neq 0$, (v)(b) follows by choosing $(q_1, \hdots, q_{t-2}, q_{t-1}, \hdots , q_{t+s}) = (0, \hdots ,-1, 1, \hdots, 0)$. Thus, it follows that $\D$ is a split metacyclic data set. Further, a direct application of Theorem~\ref{thm:main} would show that $\D$ indeed encodes the weak conjugacy represented by $(H,(\G,\F))$, as desired. 

The converse follows immediately from Remark~\ref{rem:ext_weak_conj} and Proposition~\ref{prop:main}.
\end{proof}

We now provide a couple of examples of dihedral actions along with their realizations.

\begin{example}
Consider the $ \mathbb{Z}_3 \rtimes_{-1} \mathbb{Z}_2$-action $\langle \F, \G \rangle$ on $S_3$ illustrated in Figure~\ref{fig:D6_S3} below, where $$D_F = (3,1;(1,3),(2,3)) \text{ and }D_G = (2,1;(1,2),(1,2),(1,2),(1,2)).$$
\begin{figure}[H]
	\centering
	\labellist
	\small
	\pinlabel $\F$ at 173 52
	\pinlabel $\G$ at 90 180
	\pinlabel $\frac{2 \pi}{3}$ at 172 35
	\pinlabel $\pi$ at 102 168
	\endlabellist
	\includegraphics[width=27ex]{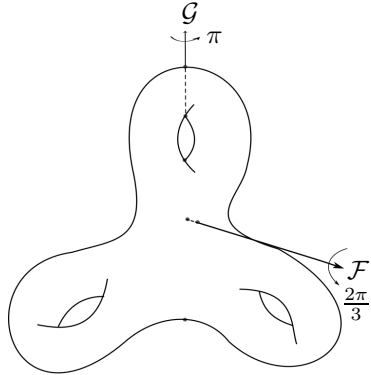}
	\caption{Realization of a $D_6$-action in $\Mod(S_3)$.}
	\label{fig:D6_S3}
	\end{figure}
\noindent The weak conjugacy class of the action $(\langle \F ,\G \rangle, (\G,\F))$ is encoded by \begin{gather*}
\D = ((2 \cdot 3,-1),0;[(1,2),(0,1),2],[(1,2),(0,1),2],[(1,2),(0,1),2], \\ [(1,2),(1,3),2],[(0,1),(2,3),3]).
\end{gather*}
\end{example}

\begin{example}
Consider the $\mathbb{Z}_4 \rtimes_{-1} \mathbb{Z}_2$-actions $\langle \F, \G \rangle$ and $\langle \F, \G' \rangle$ on $S_3$ illustrated in Figure~\ref{fig:D8_S3} below, where $D_F = (4,0;(1,4),(3,4),(1,4),(3,4))$, $D_G = (2,1;(1,2),(1,2),(1,2),(1,2))$, and $D_{G'} = (2,2,1;).$ 
\begin{figure}[H]
	\centering
	\labellist
     \small
	\pinlabel $\F$ at 110 455
	\pinlabel $\G$ at 240 150
	\pinlabel $\frac{\pi}{2}$ at 135 430
	\pinlabel $\pi$ at 225 125
	\pinlabel $\G'$ at -10 100
	\pinlabel $\pi$ at 25 80
	\endlabellist
	\includegraphics[width=20ex]{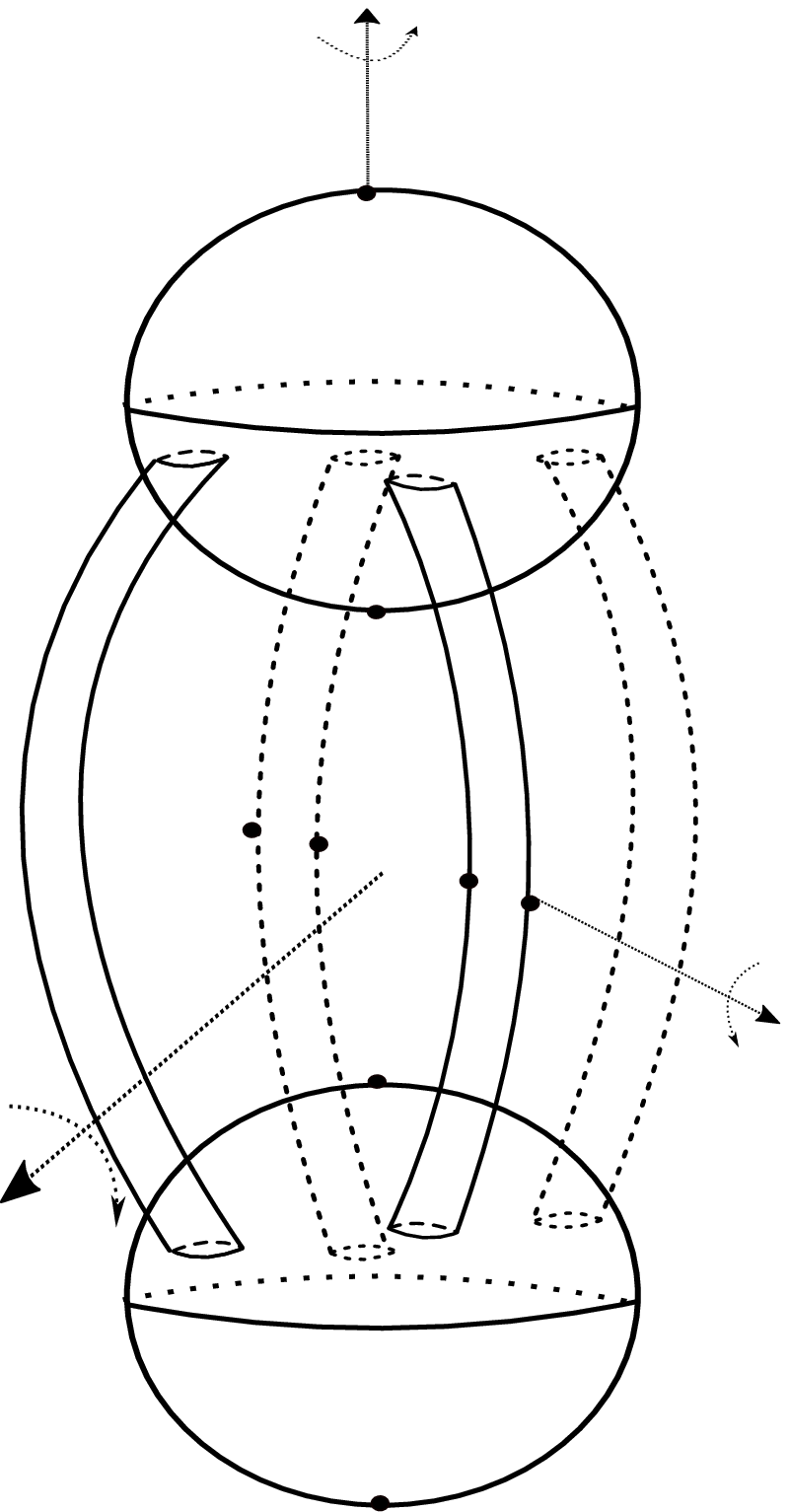}
	\caption{Realization of a $D_8$-action in $\Mod(S_3)$.}
	\label{fig:D8_S3}
	\end{figure}
	
	\noindent The weak conjugacy classes $(\langle \F ,\G \rangle, (\G,\F))$ and $(\langle \F ,\G' \rangle, (\G',\F))$ are encoded by $$((2 \cdot 4,-1),0;[(1,2),(0,1),2],[(1,2),(0,1),2],[(0,1),(1,4),4],[(0,1),(3,4),4])$$ and $$((2 \cdot 4,-1),0;[(1,2),(1,4),2],[(1,2),(1,4),2],[(0,1),(1,4),4],[(0,1),(3,4),4]),$$ respectively.
\end{example}

\subsection{Generalized quaternions}
\label{subsec:quater}
For $n \geq 2$, the generalized quaternion group $Q_{2^{n+1}}$ is a metacyclic group of order $2^{n+1} $ that admits the presentation
$$\langle x,y \, | \, x^{2^n} = y^4 =1, \,  x^{2^{n-1}} = y^2, \, y^{-1}xy = x^{-1}\rangle.$$ 

\begin{rem} 
\label{rem:no_free_inv}
Let $\D$ be a split metacyclic data set of genus $g$, degree $4 \cdot 2^n$ and twist factor $-1$ (as in Definition~\ref{defn:meta_cyc_dataset}) encoding a weak conjugacy class represented by $(H,(\G,\F))$. Suppose that $\D$ has the property that $[(c_{j1}, n_{j1}),(c_{j2}, n_{j2}), n_j] = [(1,2), (1,2), 2]$, for some $1 \leq j \leq \ell$. Then it follows from the proof of Proposition~\ref{prop:main} that under the epimorphism $\phi_H : \pi_1^{orb}(\O_H) \to H$ which preserves the order of torsion elements, the tuple $[(1,2), (1,2), 2]$ would correspond to an involution $\G^{2}\F^{2^{n-1}}\in H$ which defines a non-free action on $S_g$.
\end{rem} 

\noindent Remark~\ref{rem:no_free_inv} motivates the following definition. 
\begin{defn}
\label{defn:gen_quat_dataset}
A \textit{quaternionic data set} is a split metacyclic data set of degree $4 \cdot 2^n$ that has the form
$$\D = ((4 \cdot 2^n,-1), g_0; [(c_{11}, n_{11}),(c_{12}, n_{12}), n_1], \hdots ,[(c_{\ell 1}, n_{\ell 1}), (c_{\ell 2}, n_{\ell 2}), n_\ell] ),$$ such that $[(c_{j1}, n_{j1}),(c_{j2}, n_{j2}), n_j] \neq [(1,2), (1,2), 2]$, for $1 \leq j \leq \ell.$ 
\end{defn}

\begin{prop}
\label{prop:quat_action}
For $g,n \geq 2$, quaternionic data sets of genus $2g-1$ and degree $4 \cdot 2^n$ correspond to $Q_{2^{n+1}}$-actions on $S_g$. 
\end{prop}

\begin{proof}
Suppose that there exists an action of $H=Q_{2^{n+1}}$ on $S_g$. By Lemma~\ref{Harvey Condition}, there exists an epimorphism $\phi_H : \pi_1^{orb}(\O_H) \to H$
$$\displaystyle \xi_i \xmapsto{\phi_H} y^{c_{i1} \frac{m}{n_{i1}}} x^{c_{i2} \frac{n}{n_{i2}}}, \text{ for } 1 \leq i \leq \ell,$$ that is order-preserving on torsion elements. Let $H' = \Z_{2^n} \rtimes_{-1} \Z_4$. Since the canonical projection $q: H' \to H (\cong H'/\Z_2)$ which preserves the order of torsion elements on $H'\setminus \ker q$, the map $\phi_H$ naturally factors via $q$. Thus, as there are exactly two possible choices for $\phi_H \vert_{\{\xi_{i}: 1 \leq i \leq \ell\}}$ that preserves the order, at least one of which yields an action $H'$ on $S_{g'}$ (for some $g' > g$). A weak conjugacy class associated with this action is encoded by a split metacyclic data set of genus $g'$ and degree $2^{n+2} = 4\cdot 2^n$, which has one of the following forms
\begin{gather*}
((4 \cdot 2^n,-1), g_0; [(c_{11}, n_{11}),(c_{12}, n_{12}), n_1], \hdots ,[(c_{\ell 1}, n_{\ell 1}), (c_{\ell 2}, n_{\ell 2}), n_\ell] ) \\ \text{ or } \\ 
(4 \cdot 2^n,-1), g_0; [(c_{11}, n_{11}),(c_{12}, n_{12}), n_1], \hdots ,[(c_{\ell 1}', n_{\ell 1}'), (c_{\ell 2}', n_{\ell 2}'), n_\ell] ),
\end{gather*} 
where $c_{\ell 1}' \frac{4}{n_{\ell 1}'} \equiv c_{\ell 1} \frac{4}{n_{\ell 1}} +2 \pmod{4} \text{ and } c_{\ell 2}' \frac{2^n}{n_{\ell2}'} \equiv c_{\ell 2} \frac{2^n}{n_{\ell 2}} + 2^{n-1} \pmod{2^n}.$
Further, since $\ker q \cong \Z_2$ and $q$ preserves the orders of all $x\in H' \setminus \ker q$, it follows that $\ker q$ acts freely on $S_{g'}$. Hence, it follows that $g' = 2g-1$ and further by Remark~\ref{rem:no_free_inv}, both (possible) tuples cannot contain a triple of the type $[(1,2), (1,2), 2]$.

Conversely, if there exists a quaternionic data set $\D$ of genus $g' = 2g-1$ as in Definition~\ref{defn:gen_quat_dataset}. Then we obtain an epimorphism $\phi_{H'}: \pi_1^{orb}(\O_{H'}) \to H'$ which preserves the order of torsion elements, when composed with canonical projection $q: H' \to H$, yields an epimorphism $\phi_H: \pi_1^{orb}(\O_{H}) \to H$ which preserves the order of torsion elements, where $\pi_1^{orb}(\O_{H'}) = \pi_1^{orb}(\O_{H})$. Further, as $\D$ does not contain a triple of type $[(1,2), (1,2), 2]$, $\ker q$ acts freely on $S_{g'}$, thereby yielding an action of $Q_{2^{n+1}}$ on $S_g$, where $g' = 2g-1$.
\end{proof}

\begin{rem}
A crucial step in the proof (of Proposition~\ref{prop:quat_action}) is the establishment of the fact that the canonical projection $q : \Z_{2^n} \rtimes_{-1} \Z_4 \rightarrow Q_{2^{n+1}}$ is order-preserving. However, it is interesting to note that this fact does not generalize to arbitrary metacyclic groups~\cite{CEH} arising as quotients of split metacyclic groups. This motivates the study of finite non-split metacyclic actions on surfaces, which we plan to undertake in future works. 
\end{rem}

\begin{example}
The split metacyclic data set in Example~\ref{exm:Z4_Z4_S3} is quaternionic. Hence, this represents the weak conjugacy class of an induced $Q_{8}$-action on $S_3$.
\end{example}

\subsection{Lifting cyclic subgroups of mapping classes to split metacyclic groups} 
\label{subsec:liftable}
For $n,g \geq 2$, let $p : S_{\tilde{g}} \to S_g$ be a covering map (that is possibly branched) with deck transformation group $\Z_n \cong \langle \F \rangle$. Let $\LMod_p(S_g)$ (resp. $\SMod_p(S_{\tilde{g}})$) denote the liftable (resp. symmetric) mapping class groups of $S_g$ (resp. $S_{\tilde{g}}$) under $p$.

\begin{rem}
\label{rem:bes}
From Birman-Hilden theory~\cite{JB}, we have the exact sequence
\[1 \to \langle F \rangle \to \SMod_p(S_{\tilde{g}}) \to \LMod_p(S_g) \to 1. \tag{B}\]
Let $G \in \Mod(S_g)$ be of finite order. Then $G \in \LMod_p(S_g)$ if and only if $G$ has a lift $\tilde{G} \in \SMod_p(S_{\tilde{g}})$ of finite order so that the sequence $(B)$ yields a sequence of the form 
$$1 \to \langle F \rangle \to \langle F, \tilde{G} \rangle \to \langle G \rangle \to 1.$$ Thus, $G \in \LMod_p(S_g)$ if and only if for any lift $\tilde{G}$ of $G$, $\langle G \rangle$ lifts under $p$ to a metacyclic group $\langle F, \tilde{G} \rangle$.
\end{rem}

\noindent In the following corollary, we characterize the finite cyclic subgroups in $\Mod(S_{0,3})$ that lift to finite split metacyclic groups under branched covers induced by irreducible cyclic actions.

\begin{cor}
\label{cor:irred_F}
For $g,n \geq 2$, let $p : S_g \to S_{0,3}$ be a cover with deck transformation group $\langle \F \rangle$ with $D_F = (n,0;(c_1,n_1),(c_2,n_2),(c_3,n_3)).$ Then a $G' \in \Mod(S_{0,3})$ of order $m$ has a conjugate $G \in \LMod_p(S_{0,3})$ with a lift $\tilde{G} \in \SMod_p(S_g)$ such that $\langle F, \tilde{G} \rangle \cong \Z_n \rtimes_k \Z_m$ if and only if one of the following conditions hold.
\begin{enumerate}[(a)]
\item $D_{F} = (n,0;(c_1,n_1),(c_2,n),(c_2 k,n))$ for some $k \in \Z_n^{\times}$ such that $k^2 \equiv 1 \pmod{n}$.
\item $D_{F} = (n,0;(c_1,n),(c_1 k,n),(c_1 k^2 ,n))$ for some $k \in \Z_n^{\times}$ such that $k^3 \equiv 1 \pmod{n}$.
\end{enumerate}
\end{cor}

\begin{proof}
Suppose that $G' \in \Mod(S_{0,3})$ has a conjugate $G \in \LMod_p(S_{0,3})$ with a lift $\tilde{G} \in \SMod_p(S_g)$ such that $H =\langle F, \tilde{G} \rangle \cong \Z_n \rtimes_k \Z_m$. First, we claim that the $n_i$, for $1 \leq i \leq 3$, are not distinct. Suppose that we assume on the contrary that the $n_i$, for $1 \leq i \leq 3$, are indeed distinct. Since $\G' \in \Aut(\O_{\langle \F \rangle})$ and $|\G'| >1$, it would have to fix all three cone points of $\O_{\langle \F \rangle}$, which contradicts the fact that any nontrivial automorphism of the sphere has exactly two fixed points. Thus, the following two cases arise.
	
	\textit{Case} 1: $n_2 = n_3 = n \neq n_1$. In this case, $\G'$ fixes the cone point, say of order $n_1$, and should permute the remaining 2 cone points of orders $n_2$ and $n_3$. This implies that $D_{F}$ takes the form in condition (a) in our hypothesis (by Definition~\ref{defn:ind_auto}), and hence $H = \langle F, \tilde{G} \rangle \cong \Z_n \rtimes_k \Z_2$. 
	
	\textit{Case} 2: $n_i = n, 1 \leq i \leq 3$. In this case, if $\G'$ permutes all the three cone points cyclically, then $D_F$ takes the form in condition (b) in our hypothesis, and hence $H  \cong \Z_n \rtimes_k \Z_3$.  Alternatively, $\G'$ could also fix a cone point of order $n$ and permute the remaining 2 cone points, in which case, $D_F$ will take the form in condition (a).  
	
	Conversely, if $D_F = (n,0;(c_1,n_1),(c_2,n),(c_2k,n))$ for some $k \in \Z_n^{\times}$ such that $k^2 \equiv 1 \pmod{n}$. Up to conjugacy, let $\G' \in \Aut( \O_{\langle \F \rangle})$ be an involution so that $\G'$ maps the cone point represented by $(c_2,n)$ to the cone point represented by $(c_2 \cdot k,n)$. To prove our assertion, it would suffice to show the existence of an involution $\G \in \Homeo^+(S_g)$ that induces $\G'$. This amounts to showing that there exists a split metacyclic data set $\D$ of degree $2 \cdot n$ with twist factor $k$ encoding the weak conjugacy class $(H,(\G,\F))$ so that $D_G$ has degree $2$. Consider the tuple
	$((2 \cdot n,k),0; [(1,2),(0,1),2],[(1,2),(n-c_2,n),2 n_1],[(0,1),(c_2,n),n])$. By simple computation would reveal that conditions (i) - (iv) of Definition~\ref{defn:meta_cyc_dataset} hold true. Condition (v) is true by taking $v = 1$, $(p_1,p_2,p_3) = (1,0,0)$ and $(q_1,q_2,q_3) = (0,0,w)$ such that $w c_2 \equiv 1 \pmod{n}$, which proves our claim. 
	
	For the case when $D_{F} = (n,0;(c_1,n),(c_1 k ,n),(c_1k^2 ,n))$ for some $k \in \Z_n^{\times}$ such that $k^3 \equiv 1 \pmod{n}$, let $\G' \in \Aut( \O_{\langle \F \rangle})$ be of order $3$ so that for $1 \leq i \leq 2$, $\G'^i$ maps the cone point represented by $(c_1,n)$ to the cone point represented by $(c_1 k^i,n)$. By similar argument as above, we can show that the tuple $((3 \cdot n,k),0; [(1,3),(0,1),3],[(2,3),(n-c_1,n),3],[(0,1),(c_1,n),n])$ forms a split metacyclic data set of degree $3 \cdot n$ with twist factor $k$.
\end{proof}

\begin{example}
For $i=1,2$, consider the branched cover $p: S_3 \to \O_{\langle \F_i \rangle} (\approx S_{0,3})$, where $D_{F_1} = (8,0;(1,4),(1,8),(5,8))$ and $D_{F_2} = (8,0;(3,4),(3,8),(7,8))$. Then (up to conjugacy) the order-$2$ mapping class $G \in \LMod_p(S_{0,3})$ represented by an automorphism ${\G} \in \text{Aut}_5(S_{0,3})$, that permutes two cone points of order 8 and fixes order 4 cone point, lifts to a $\tilde{G} \in \SMod_p(S_3)$ with $D_{\tilde{G}} = (2,1;((1,2),4))$ such that $\langle F_i , \tilde{G} \rangle  \cong \mathbb{Z}_8 \rtimes_5 \mathbb{Z}_2$. Moreover, the weak conjugacy class of $(\langle \F_i , \tilde{\G} \rangle, (\tilde{\G}, \F_i))$, for $i=1,2$, is encoded by 
\begin{gather*} 
((2 \cdot 8,5),0;[(1,2),(0,1),2],[(1,2),(7,8),8], [(0,1),(1,8),8]) \text{ and } \\
((2 \cdot 8,5),0;[(1,2),(0,1),2],[(1,2),(1,8),8], [(0,1),(7,8),8]),
\end{gather*}
respectively. The geometric realization of these actions is illustrated in Figure~\ref{fig:Z8_5_Z2_S3} below, where for each $i$, the action $\F_i$ is realized by the rotation of a polygon of type $\P_{F_i}$ described in Theorem~\ref{res:1}. 
\begin{figure}[H]
	\centering
	\labellist
    \tiny
	\pinlabel $\G$ at 145 270
	\pinlabel $\pi$ at 160 250
	\pinlabel $(1,8)$ at -23 45
	\pinlabel $(1,4)$ at 140 20
	\pinlabel $(1,4)$ at 140 140
	\pinlabel $(5,8)$ at 298 55
	\pinlabel $\G$ at 490 270
	\pinlabel $\pi$ at 504 250
	\pinlabel $(3,8)$ at 323 38
	\pinlabel $(3,4)$ at 485 20
	\pinlabel $(3,4)$ at 485 140
	\pinlabel $(7,8)$ at 645 50
	\endlabellist
	\includegraphics[width=55ex]{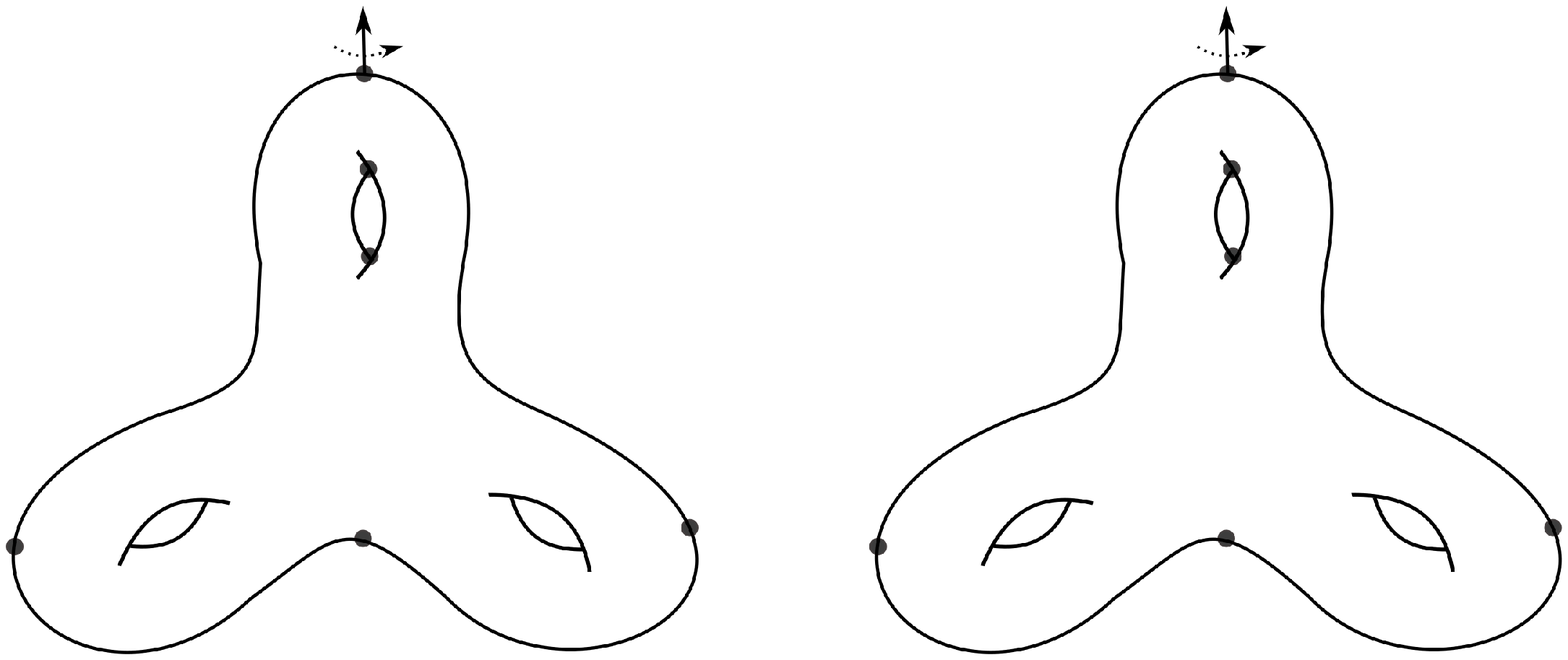}
	\caption{The realizations of two distinct $\mathbb{Z}_8 \rtimes_5 \mathbb{Z}_2$-actions on $S_3$.}
	\label{fig:Z8_5_Z2_S3}
	\end{figure}
\end{example}

\begin{prop}
For $g,n \geq 2$, let $p : S_{n(g-1)+1} \rightarrow S_g$ be a regular cover with deck transformation group $\Z_n \cong \langle \F \rangle$. Then any involution $G' \in \Mod(S_g)$ has a conjugate $G \in \LMod_p(S_g)$ with a lift $\tilde{G} \in \SMod_p(S_{n(g-1)+1})$ such that $\langle F, \tilde{G} \rangle \cong D_{2n}$.
\end{prop}

\begin{proof}
Let $G' \in \Mod(S_g)$ be an involution. When $\G'$ generates a free action on $S_g$, it is easy to see that $(\langle F, \tilde{G} \rangle, (\tilde{G}, F))$ represents a weak conjugacy class in $\Mod(S_{n(g-1)+1})$ with $\langle F, \tilde{G} \rangle \cong D_{2n}$. Now, we assume that $\G'$ generates a non-free action with $D_{G'}= (2, g_0; ((1,2),t))$. By Theorem~\ref{thm:main} and Remark~\ref{rem:bes}, it suffices to show that there exists a dihedral data set $\D$ of degree $2\cdot n$ and genus $n(g-1)+1$ representing the weak conjugacy class of $(\langle F, \tilde{G} \rangle, (\tilde{G}, F))$. When $g_0 \geq 1$, we take $\D$ to be the tuple $$((2 \cdot n,-1),g_0; \underbrace{[(1,2),(0,1),2], \hdots ,[(1,2),(0,1),2]}_{t ~ \text{times}},$$ and when $g_0 = 0$, $t \geq 4$, and so we take $\D$ to be the tuple 
\begin{gather*}
((2 \cdot n,-1),0; \underbrace{[(1,2),(0,1),2], \hdots, [(1,2),(0,1),2]}_{t-2 ~ \text{times}}, \\ [(1,2),(1,n),2], [(1,2),(1,n),2]).
\end{gather*}
It is an easy computation to check that $\D$ satisfies conditions (i)-(iv) of Definition~\ref{defn:meta_cyc_dataset} in both cases. When $g_0 = 0$, taking $v=1$, $$(p_1, \hdots , p_t) =   (1,0, \hdots , 0), \text{ and } (q_1,\ldots,q_{t}) = (0, \hdots, 0,1,1,0)$$ we obtain condition (v). Moreover, when $g_0 = 1$, we take $v=1$, 
$$(p_1, \hdots , p_t) = (1,0, \hdots , 0), \text{ and } (q_1,\ldots,q_{t}) = (0, \hdots, 0),$$ thereby verifying condition (vi). Thus, we have shown that $\D$ is a dihedral data set as desired. Finally, it follows from Theorem~\ref{thm:main} that $\D$ encodes the weak conjugacy class of $(\langle F, \tilde{G} \rangle, (\tilde{G}, F))$. 
\end{proof}

\noindent Note that the same $\mathbb{Z}_{2}$-action can lift to multiple non-isomorphic groups under a regular cyclic cover. We illustrate this phenomenon in the following example.

\begin{example}
Let $p : S_5 \rightarrow S_2$ be a regular $4$-sheeted cover with deck transformation group $\Z_4= \langle \F \rangle$ as illustrated in Figure~\ref{fig:lift_Z2_S5} below. 
\begin{figure}[H]
	\centering
	\labellist
	\small
	\pinlabel $\F$ at 85 238
	\pinlabel $\tilde{\G_2}$ at 195 108
	\pinlabel $\frac{\pi}{2}$ at 105 225
	\pinlabel $\pi$ at 225 100
	\pinlabel $\tilde{\G_1}$ at 195 65
	\pinlabel $\pi$ at 185 85
	\pinlabel $\F \tilde{\G_1}$ at -8 40
	\pinlabel $\pi$ at -6 60
	\endlabellist
	\includegraphics[width=55ex]{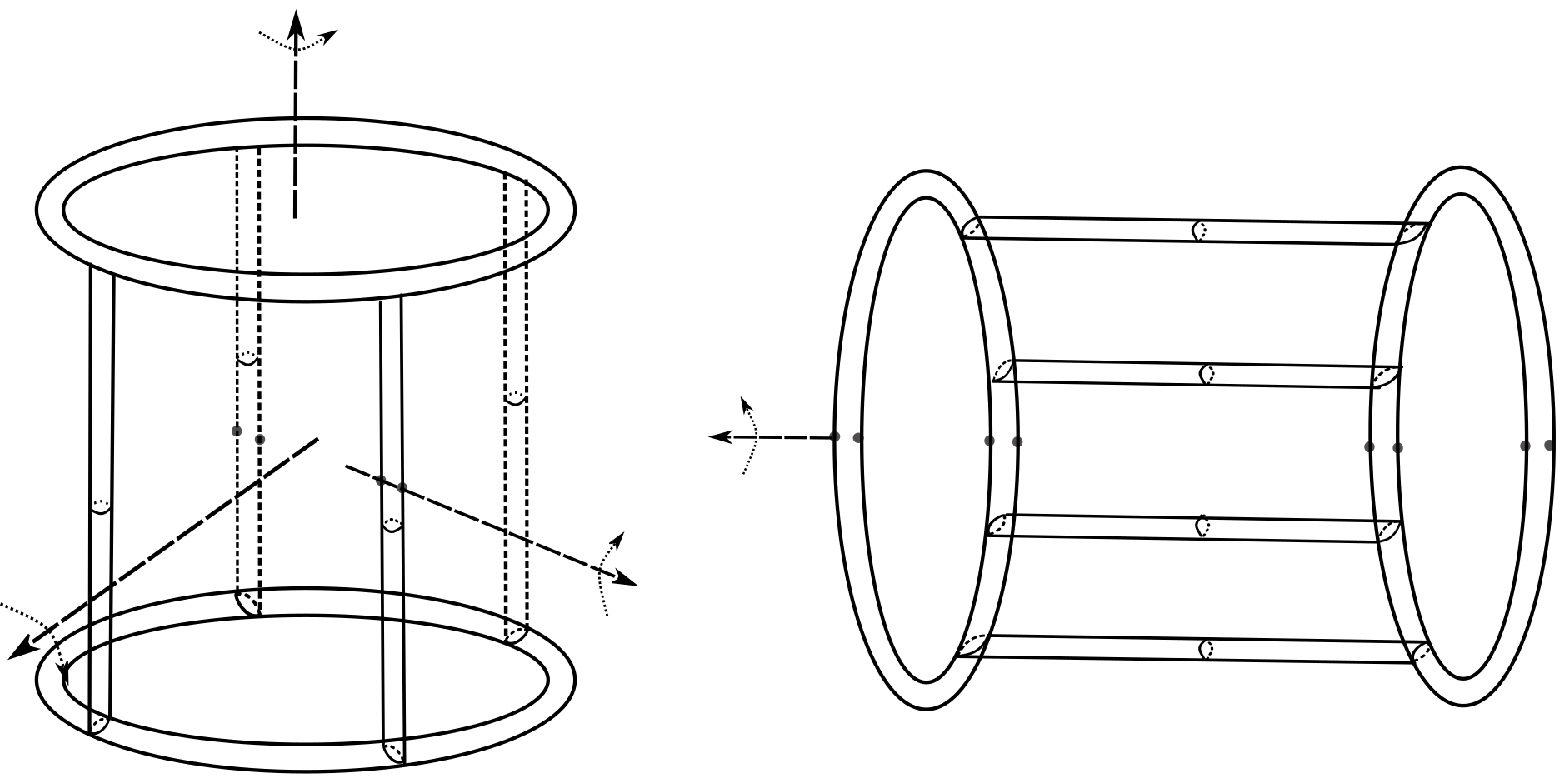}
	\caption{Two distinct lifts $\tilde{G}_1,\tilde{G}_2 \in \SMod(S_5)$ of an involution in $G \in \Mod(S_2)$.}
	\label{fig:lift_Z2_S5}
	\end{figure}
\noindent The involution $G \in \Mod(S_2)$ with $D_G = (2,1;(1,2),(1,2))$ has two distinct lifts $\tilde{G}_1,\tilde{G}_2 \in \SMod_p(S_5)$ (as indicated) such that $\langle F, \tilde{G}_1 \rangle \cong D_{8}$ and $\langle F, \tilde{G}_2 \rangle \cong \mathbb{Z}_2 \times \mathbb{Z}_4$. Note that the weak conjugacy class of $(\langle \F, \tilde{\G}_1 \rangle, (\tilde{\G}_1,\F))$ is represented by $((2 \cdot 4,-1),1;[(1,2),(0,1),2],[(1,2),(0,1),2])$.
\end{example}

\noindent The following proposition provides a sufficient condition for the liftability of $\Z_m$-actions whose corresponding orbifolds are spheres with a cone point of order $m$. 

\begin{prop}
\label{prop:lift_sph_ordern_cpt}
For $g,n \geq 2$, let $p : S_{n(g-1)+1} \rightarrow S_g$ be a regular $n$-sheeted cover with deck transformation group $\Z_n \cong \langle \F \rangle$. Let $G' \in \Mod(S_g)$ of be order $m$ such that $D_{G'} = (m,0;(c_1,m_1), \hdots , (c_{\ell},m_{\ell}))$ with $m_{\ell} = m$ (say).  Then $G'$ has a conjugate $G \in \LMod_{p}(S_g)$ with a lift $\tilde{G} \in \SMod_p(S_{n(g-1)+1})$ such that $\langle F, \tilde{G} \rangle \cong \Z_n \rtimes_k \Z_m$ if the following conditions hold.
\begin{enumerate}[(a)]
\item  There exists $a_1, \hdots , a_{\ell -1} \in \Z$, and $k \in \Z_n^{\times}$, $k^m \equiv 1 \pmod{n}$ such that
$$\displaystyle \sum_{i=1}^{\ell -1} a_i (k^{c_i \frac{m}{m_i}} -1) \prod_{s=i+1}^{\ell -1} k^{c_s \frac{m}{m_s}} \equiv 0 \pmod{n}.$$
\item  Let $d_i \frac{n}{n_i} \equiv a_i (k^{c_i \frac{m}{m_i}} -1) \pmod{n}$, for $1\leq i \leq \ell -1$, where $\gcd(d_i,n_i) = 1$ and $n_i \mid n$. Then we have $$\displaystyle \mathrm{lcm}(n_1, n_2, \hdots , n_{\ell -1} )= n.$$
\end{enumerate}
\end{prop}

\begin{proof}
By Theorem~\ref{thm:main} and Remark~\ref{rem:bes}, it suffices to show that the tuple 
\begin{gather*}\D = ((m \cdot n, k),0; [(c_1,m_1),(d_1,n_1),m_1], \ldots, \\ [(c_{\ell -1},m_{\ell -1}),(d_{\ell -1},n_{\ell -1}),m_{\ell -1}],  [(c_{\ell},m_{\ell}),(0,1),m_{\ell}]) 
\end{gather*}
forms a split metacyclic data set of genus $n(g-1)+1$ that represents the weak conjugacy class of $(\langle F, \tilde{G}\rangle, (\tilde{G},F))$ for some lift $\tilde{G}$ of $G$ under $p$. It can be verified easily that $\D$ satisfies conditions (i)-(iii) of Definition~\ref{defn:meta_cyc_dataset}, and further, condition (iv) follows from condition (a) in our hypothesis. By taking $v=1$, 
$(p_1, \hdots ,p_{\ell}) = (0, \hdots ,0,w)$ such that $w c_{\ell} \equiv 1 \pmod{m}$, we see that condition (v)(a) holds. Finally, condition (v)(b) follows from condition (b) in our hypothesis, and our assertion follows.
\end{proof}

\noindent Using similar arguments, we can show the following. 

\begin{prop}
\label{prop:lift_sph_no_ordern_cpt}
For $g,n \geq 2$, let $p : S_{n(g-1)+1} \rightarrow S_g$ be a regular $n$-sheeted cover with deck transformation group $\Z_n \cong \langle \F \rangle$. Let $G' \in \Mod(S_g)$ be of order $m$ such that $D_{G'} = (m,0;(c_1,m_1), \hdots , (c_{\ell},m_{\ell}))$ with $m_{i} \neq m$, for $1 \leq i \leq \ell$.  Then $G'$ has a conjugate $G \in \LMod_{p}(S_g)$ with a lift  $\tilde{G} \in \SMod_p(S_{n(g-1)+1})$ such that $\langle F, \tilde{G} \rangle \cong \Z_n \rtimes_k \Z_m$ if following conditions hold.
\begin{enumerate}[(i)]
\item  There exists $a_1, \hdots , a_{\ell} \in \Z$, and $k \in \Z_n^{\times}$, $k^m \equiv 1 \pmod{n}$ such that
$$\displaystyle \sum_{i=1}^{\ell} a_i (k^{c_i \frac{m}{m_i}} -1) \prod_{s=i+1}^{\ell} k^{c_s \frac{m}{m_s}} \equiv 0 \pmod{n}.$$
\item There exists $(p_1, \ldots , p_{\ell v}), (q_1, \ldots, q_{\ell v}) \in \Z^{\ell v}$ and $v \in \mathbb{N}$ such that condition (v)(b) of Definition~\ref{defn:meta_cyc_dataset} holds, where for $1 \leq i \leq\ell$, we have
$$c_{i1}\frac{m}{n_{i1}} \equiv c_i\frac{m}{m_i} \pmod{m} \text{ and } c_{i2}\frac{n}{n_{i2}} \equiv a_i(k^{c_i\frac{m}{m_i}}-1) \pmod{n}.$$
\end{enumerate}
\end{prop}

\noindent A consequence of Propositions~\ref{prop:lift_sph_ordern_cpt}-\ref{prop:lift_sph_no_ordern_cpt} is the following. 

\begin{cor}
For $g\geq 2$ and prime $n$, let $p : S_{n(g-1)+1} \rightarrow S_g$ be a regular $n$-sheeted cover with deck transformation group $\Z_n \cong \langle \F \rangle$. Let $G' \in \Mod(S_g)$ be of order $m$ such that the genus of $\O_{\langle \G' \rangle}$ is zero. Then $G'$ has a conjugate $G \in \LMod_{p}(S_g)$ with a lift  $\tilde{G} \in \SMod_p(S_{n(g-1)+1})$ such that $\langle F, \tilde{G} \rangle \cong \Z_n \rtimes_k \Z_m$ if there exists $k \in \Z_n^{\times}$ such that $|k| =m$.
\end{cor}

\begin{proof}
Let $D_{G'} = (m,0;(c_1,m_1), \hdots , (c_{\ell},m_{\ell}))$. First, let us assume (without loss of generality) that $m_{\ell} = m$. By choosing
$$(a_1,\ldots,a_{\ell-1}) = (0,\ldots,0,1,-(k^{c_{\ell-2}\frac{m}{m_{\ell-2}}}-1)\cdot k^{c_{\ell-1}\frac{m}{m_{\ell-1}}}\cdot(k^{c_{\ell-1}\frac{m}{m_{\ell-1}}}-1)^{-1}),$$ we see that condition (i) of Proposition~\ref{prop:lift_sph_ordern_cpt} holds true. Moreover, since $|k|=m$, we have $\gcd((k^{c_{\ell-2}\frac{m}{m_{\ell-2}}}-1),n)=1$, and so condition (ii) also holds, and our assertion follows.

Similarly, for the case when each $m_{i} <m$ for $1 \leq i \leq \ell$, the result follows by taking $$(a_1,\ldots,a_{\ell}) = (0,\ldots,0,1,-(k^{c_{\ell-1}\frac{m}{m_{\ell-1}}}-1)\cdot k^{c_{\ell}\frac{m}{m_{\ell}}}\cdot(k^{c_{\ell}\frac{m}{m_{\ell}}}-1)^{-1}),$$ and applying Proposition~\ref{prop:lift_sph_no_ordern_cpt}.

\end{proof}

 \subsection{Infinite split metacyclic subgroups of $\Mod(S_g)$} 
\label{subsec:inf_meta}
An \textit{infinite split metacyclic group} that is isomorphic to $\Z \rtimes_{-1}  \Z_{2m}$ admits a presentation of the form
 \begin{equation}
 \label{eq:inf_split}
 \langle x,y \, |\, y^{2m}=1, y^{-1}xy = x^{-1} \rangle.
 \end{equation}
\noindent In this subsection, we give an explicit construction of an infinite metacyclic subgroup isomorphic to $\Z \rtimes_{-1}  \Z_{2m}$ of $\Mod(S_g)$. Let $T_c \in \Mod(S_g)$ denote the left-handed Dehn twist about a simple closed curve $c$ in $S_g$. A \textit{root of $T_c$ of degree $s$} is an $F \in \Mod(S_g)$ such that $F^s = T_c$. In the following lemma, by using some basic properties of Dehn twists~\cite[Chapter 3]{FM}, we show that a root of Dehn twist cannot generate an infinite split metacyclic group that admits a presentation as in~(\ref{eq:inf_split}).
\begin{lemma}
	\label{lem:no_root_inf_split}
	For $g \geq 2$, no root of $T_c$ is a generator of any infinite split metacyclic subgroup of $\Mod(S_g)$ of type in Equation~\ref{eq:inf_split}.
\end{lemma}
\begin{proof}
	Let $F$ be a root of $T_c$ of degree $s$. Suppose we assume on the contrary that for some $g \geq 2$, there exists an infinite split metacyclic subgroup $H \cong \Z \rtimes_{-1} \Z_{2m}$ of $\Mod(S_g)$ that admits the presentation
	$$H = \langle F,G \, |\, G^{2m}=1, G^{-1}FG = F^{-1} \rangle.$$
	
	First, we consider the case when $s=1$, that is, $F=T_c$. Then we have that 
	$$G^{-1}T_cG= T_c^{-1} \implies T_{G^{-1}(c)}= T^{-1}_c,$$ which is impossible. Thus, we have that $H \neq \langle G, T_c \rangle$, which contradicts our assumption.
	
	For $s >1$, suppose that $H = \langle F,G \rangle$. Then the subgroup $\langle F^s, G \rangle$ of $H$ would also be a split metacyclic group. Since $F^s = T_c$, this would contradict our conclusion in the previous case, and so our assertion follows.
\end{proof}

\noindent By a \textit{multitwist} in $\Mod(S_g)$, we mean a finite product of powers of commuting Dehn twists. In view of Lemma~\ref{lem:no_root_inf_split}, a natural question that arises is whether a multitwist in $\Mod(S_g)$ can generate an infinite split metacyclic group. In the following examples, we answer this question in the affirmative.
\begin{example}
	\label{eg:inf_dihed}
	Let $F' \in \Mod(S_2)$ be of order $3$ with $$D_{F'} = (3,0;((1,3),2),((2,3),2)).$$ First, we note that $\F'$ has four fixed points on $S_2$. Further, it induces a local rotation angle of $2\pi/3$ around two of these points (corresponding to the two $(1,3)$ pairs in $D_{F'}$) and rotation angle of $4\pi/3$ around the remaining two points (corresponding to the two $(2,3)$ pairs in $D_{F'}$), as indicated in Figure~\ref{fig:bound_pair}. Considering this action on two distinct copies of $S_2$, we remove invariant disks around a distinguished $(1,3)$-type fixed point and a distinguished $(2,3)$-type fixed point in each of the two copies. We now attach two annuli connecting the resulting boundary components across the two surfaces so that: 
	\begin{enumerate}[(a)]
		\item each annulus connects a pair of boundary components where the induced rotation angle is the same, as shown in Figure~\ref{fig:bound_pair} below, and further,
		\item the annulus connecting the boundary components with rotation $4\pi/3$ (with the nonseparating curve $c$) has a $1/3^{rd}$ twist, while the other (with the nonseparating curve $d$) has a $-1/3^{rd}$ twist.
	\end{enumerate}
	\begin{figure}[H]
		\centering
		\small
		\labellist
		\pinlabel $(1,3)$ at 75 -1
		\pinlabel $c$ at 200 8
		\pinlabel $d$ at 200 120
		\pinlabel $(2,3)$ at 75 130
		\pinlabel $(1,3)$ at 325 -1
		\pinlabel $(2,3)$ at 325 130
		\pinlabel $(1,3)$ at 145 100
		\pinlabel $(1,3)$ at 259 100
		\pinlabel $(2,3)$ at 145 28
		\pinlabel $(2,3)$ at 259 28
		\pinlabel $\pi$ at 420 77
		\pinlabel $\G$ at 442 79
		\endlabellist
		\includegraphics[width=70ex]{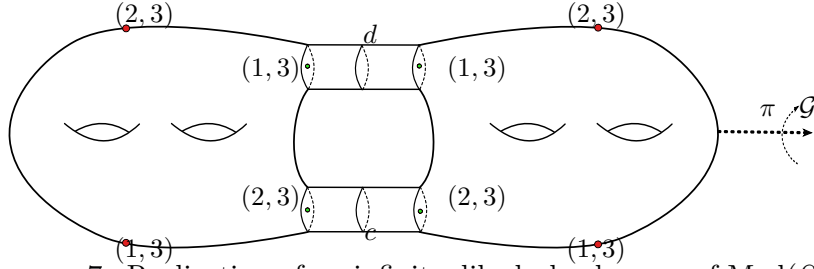}
		\caption{Realization of an infinite dihedral subgroup of $\Mod(S_5)$.}
		\label{fig:bound_pair}
	\end{figure}
	\noindent Thus, by applying the theory developed in~\cite{KP}, we obtain an $F \in \Mod(S_5)$, which is a root of the bounding pair map $T_cT_d^{-1}$ of degree $3$. Now, we consider the hyperelliptic involution $G \in \Mod(S_5)$ with $D_G = (2,0;((1,2),12))$ (also indicated in Figure~\ref{fig:bound_pair}). By our construction, it follows that $GFG^{-1} = F^{-1}$, and so we  have $\langle F, G \rangle \cong \Z \rtimes_{-1} \Z_2$.
\end{example}

\begin{example}
	\label{eg:inf_metafree}
	Let $F' \in \Mod(S_5)$ be of order $3$ with $$D_{F'} = (3,1;((1,3),2),((2,3),2)).$$ First, we note that $\F'$ has four fixed points on $S_5$. Furthermore, it induces a local rotation angle of $2\pi/3$ around two of these points (corresponding to the two $(1,3)$ pairs in $D_{F'}$) and rotation angle of $4\pi/3$ around the remaining two points (corresponding to the two $(2,3)$ pairs in $D_{F'}$), as indicated in Figure~\ref{fig:inf_metafree}. Considering this action on two distinct copies of $S_5$, we remove invariant disks around all fixed point in each of the two copies. We now attach four annuli connecting the resulting boundary components across the two surfaces so that: 
	\begin{enumerate}[(a)]
		\item each annulus connects a pair of boundary components where the induced rotation angle is the same, as shown in Figure~\ref{fig:inf_metafree} below, and further,
		\item the annulus connecting the boundary components with rotation $4\pi/3$ (with the nonseparating curve $c_1$ and $c_3$) has a $1/3^{rd}$ twist, while the other (with the nonseparating curve $c_2$ and $c_4$) has a $-1/3^{rd}$ twist.
	\end{enumerate}
	\begin{figure}[H]
		
		\centering
		\labellist
		\tiny
		\pinlabel $\G$ at 275 595
		\pinlabel $\frac{\pi}{2}$ at 300 555
		\pinlabel $(1,3)$ at 295 485
		\pinlabel $(1,3)$ at 215 460
		\pinlabel $(1,3)$ at 160 465
		\pinlabel $(2,3)$ at 350 485
		\pinlabel $(2,3)$ at 410 455
		\pinlabel $(2,3)$ at 470 467
		\pinlabel $(1,3)$ at 270 315
		\pinlabel $(1,3)$ at 330 340
		\pinlabel $(1,3)$ at 310 300
		\pinlabel $(2,3)$ at 70 300
		\pinlabel $(2,3)$ at 80 350
		\pinlabel $(2,3)$ at 115 315
		\pinlabel $c_1$ at 500 350
		\pinlabel $c_2$ at 185 340
		\pinlabel $c_3$ at 100 185
		\pinlabel $c_4$ at 345 185
		\pinlabel $(1,3)$ at 260 208
		\pinlabel $(1,3)$ at 215 225
		\pinlabel $(1,3)$ at 190 180
		\pinlabel $(2,3)$ at 420 205
		\pinlabel $(2,3)$ at 455 180
		\pinlabel $(2,3)$ at 470 233
		\pinlabel $(1,3)$ at 290 45
		\pinlabel $(1,3)$ at 300 67
		\pinlabel $(1,3)$ at 365 80
		\pinlabel $(2,3)$ at 68 65
		\pinlabel $(2,3)$ at 120 80
		\pinlabel $(2,3)$ at 135 45
		\endlabellist
		\includegraphics[width=50ex]{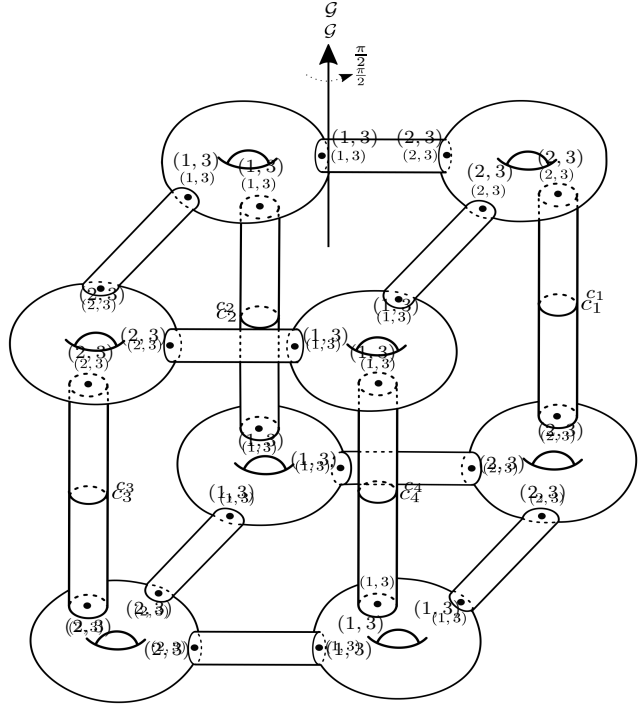}
		\caption{Realization of an infinite metacyclic subgroup of $\Mod(S_{13})$.}
		\label{fig:inf_metafree}
		
	\end{figure}
	\noindent Thus, by applying the theory developed in~\cite{KP}, we obtain an $F \in \Mod(S_{13})$, which is a root of the multitwist $T_{c_1}T_{c_2}^{-1}T_{c_3}T_{c_4}^{-1}$ of degree $3$. Now, we consider a $G \in \Mod(S_{13})$ with $D_G = (4,4,1;)$ (also indicated in Figure~\ref{fig:inf_metafree}). By our construction, as $\Z_3 \rtimes_{-1} \Z_4 \cong\langle F',G' \rangle \leq \Mod(S_5)$, where $D_{G'} = (4,2,1;)$, it follows that $GFG^{-1} = F^{-1}$, and so we  have $\langle F, G \rangle \cong \Z \rtimes_{-1} \Z_4$.
\end{example}

\noindent  Generalizing the above all constructions in Example~\ref{eg:inf_dihed} and Example~\ref{eg:inf_metafree}, we have the following.
\begin{prop}
	\label{prop:inf_split_meta}
	For $i = 1,2$, let $H_i = \langle F_i , G_i \rangle \leq \Mod(S_{g_i})$ with $H_i  \cong \Z_n \rtimes_{-1} \Z_{2m}$, such that the weak conjugacy class $(H_i,(G_i,F_i))$ is represented by a split metacyclic data set $\mathcal{D}_{H_i}$ containing a tuple $[(0,1),(a_i,n),n]$. Then there exists an infinite metacyclic subgroup $ \Mod(S_{g_1+g_2+2m-1})$ that is isomorphic to $\Z \rtimes_{-1} \Z_{2m}$ and generated by a periodic mapping class of order $2m$ and a root of a multitwist of degree $n$.
\end{prop}
\begin{proof}
	As $\mathcal{D}_{H_i}$ contains a tuple $[(0,1),(a_i,n),n]$, by Proposition~\ref{prop:main}, we have 
	\begin{gather*} D_{F_1}= (n,g_0;(c_1,n_1), \hdots , (c_s,n_s), \underbrace{(a_1,n),(n-a_1,n) , \hdots ,(a_1,n),(n-a_1,n)}_{m \text{ times }}) \\ \text{ and} \\ D_{F_2}= (n,g_0';(c_1',n_1'), \hdots , (c_t',n_t'), \underbrace{(a_2,n),(n-a_2,n) , \hdots ,(a_2,n),(n-a_2,n)}_{m \text{ times }}).
	\end{gather*} Taking inspiration from the theory developed in~\cite{KR2,KP} and Examples~\ref{eg:inf_dihed}-\ref{eg:inf_metafree}, we glue $2m$ annuli connecting the boundary components resulting from removing invariant disks around the orbit points corresponding to the pairs $(a_1,n)$ and 
	$$\begin{cases}
	(a_2,n), & \text{ if } a_2 \neq n - a_1, \text{ or}\\	
	(n-a_2,n), & \text{ if } a_2 = n - a_1.
	\end{cases}$$
This yields a degree-$n$ root $F$ of a multitwist of the form 
$$\begin{cases}
	\prod_{i=1}^{2m} T_{c_i}^{(-1)^{i+1}(a_1^{-1}+a_2^{-1})}, & \text{ if } a_2 \neq n - a_1, \text{ or}\\	
	\prod_{i=1}^{2m} T_{c_i}^{(-1)^{i+1}(a_1^{-1}+(n-a_2)^{-1})}, & \text{ if } a_2 = n - a_1,
	\end{cases}$$
where $a_ia_i^{-1} \equiv 1 \pmod{n}$ and $a_1^{-1} + a_2^{-1} \in \mathbb{Z}_n$. By considering the action $G$ obtained by performing a $2m$-compatibility on $\G_1$ and $\G_2$ (see Section~\ref{sec:hyp_str}), we see that $\langle F,G\rangle \cong \Z \rtimes_{-1} \Z_{2m}$, as desired.
\end{proof}
The group for $m = 1$ in the presentation of the infinite split metacyclic group of the type in the Equation~\ref{eq:inf_split} is known as infinite dihedral group. Here is the corollary which directly follows from Proposition~\ref{prop:inf_split_meta}
\begin{cor}
	For $g \geq 5$, there exists an infinite dihedral subgroup of $\Mod(S_g)$ that is generated by an involution and a root of a bounding pair map of degree $3$.
\end{cor}

\section{Hyperbolic structures realizing split metacyclic actions}
\label{sec:hyp_str}

We begin this section by providing an algorithm for obtaining the hyperbolic structures that realize finite split metacyclic subgroups of $\Mod(S_g)$ (up to weak conjugacy) as groups of isometries. 

\begin{enumerate}[\textit{Step} 1.]
\item Consider a weak conjugacy class represented by $(H,(\G,\F))$.

\item Use Theorem~\ref{thm:main} to determine the conjugacy classes $D_F$ (resp. $D_G$) of the generators $F$ (resp. $G$). 

\item We apply Lemma~\ref{lem:fix_metacyclic}, and Theorems~\ref{res:1}-\ref{res:2},  to obtain the hyperbolic structures that realize $H$ as a group of isometries.
\end{enumerate}

\noindent We now describe the geometric realizations of some split metacyclic actions on $S_3$ and $S_5$ represented by the split metacyclic data sets listed in Tables~\ref{tab:split_meta_dsets} and~\ref{tab:split_meta_dsets_s5} in Section~\ref{sec:classify}.

\begin{figure}[htbp]
	\centering
	\labellist
	\tiny
	\pinlabel $\G$ at 170 130
	\pinlabel $\pi$ at 175 110
	\pinlabel $(1,3)$ at -15 50
	\pinlabel $(1,3)$ at 75 18
	\pinlabel $(1,3)$ at 75 83
	\pinlabel $(2,3)$ at 275 55
	\pinlabel $(2,3)$ at 182 18
	\pinlabel $(2,3)$ at 182 83
	\endlabellist
	\includegraphics[width=32ex]{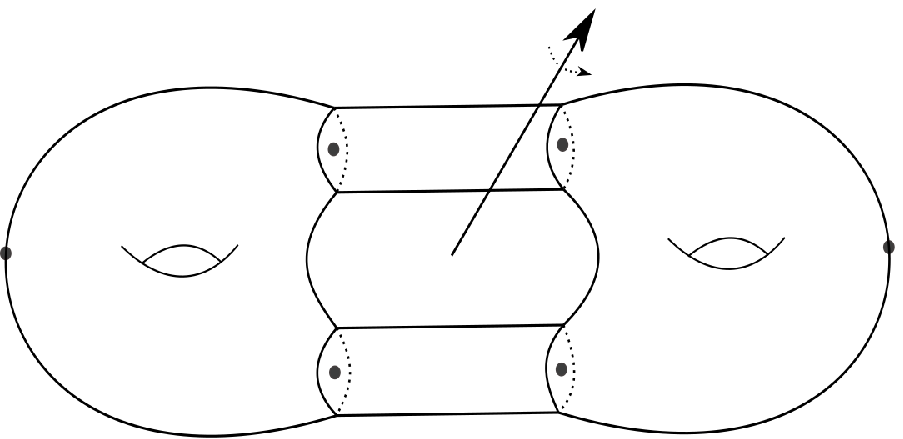}
	\caption{\small A realization of a $D_6$-action $\langle \F, \G \rangle$ on $S_3$, where $D_G = (2,2,1;)$ and $D_F = (3,1;(1,3),(2,3))$. The action $\F$ is realized through two $1$-compatibilities between two actions $F'$ and $F''$ on $S_1$ with $D_{F'} = (3,0;((1,3),3))$ and $D_{F''} = (3,0;((2,3),3)))$. The weak conjugacy class of $(\langle \F, \G\rangle, (\G,\F))$ is encoded by the first split metacyclic data set in Table~\ref{tab:split_meta_dsets}.}
	\label{fig:D6_S3_1}
	\end{figure}
	
\begin{figure}[htbp]
	\centering
	\labellist
	\tiny
	\pinlabel $\G_1$ at 175 130
	\pinlabel $\pi$ at 175 105
	\pinlabel $(1,2)$ at -5 17
	\pinlabel $(1,2)$ at -5 85
	\pinlabel $(1,4)$ at 76 17
	\pinlabel $(1,4)$ at 76 85
	\pinlabel $(1,2)$ at 261 17
	\pinlabel $(1,2)$ at 261 85
	\pinlabel $(3,4)$ at 180 17
	\pinlabel $(3,4)$ at 180 85
	\pinlabel $\G_2$ at 130 140
	\pinlabel $\pi$ at 139 120
	\endlabellist
	\includegraphics[width=32ex]{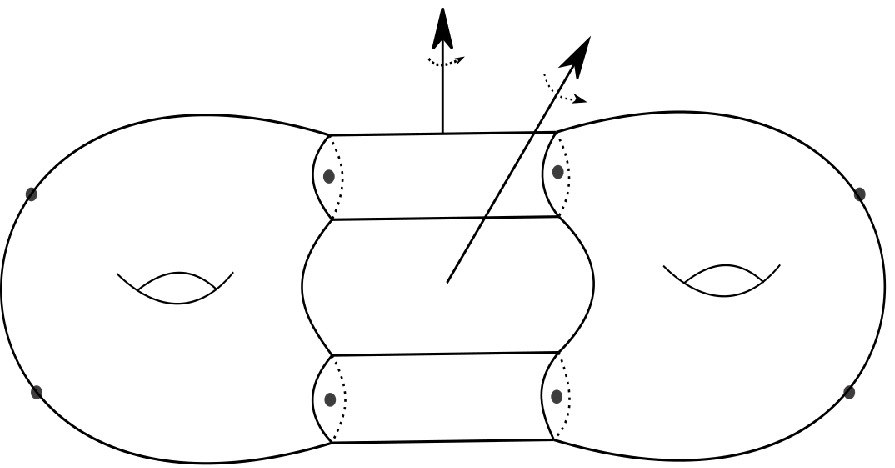}
	\caption{\small The realizations of two distinct $D_8$-actions $\langle \F, \G_1 \rangle$ and $\langle \F , \G_2 \rangle$ on $S_3$, where $D_F = (4,1;((1,2),2))$, $D_{G_1} = (2,2,1;)$, and $D_{G_2} = (2,1;((1,2),4)$. The action $\F$ is realized via two 1-compatibilities between two actions $\F'$ and $\F''$ on $S_1$, where $D_{F'} = (4,0;((1,4),2),(1,2))$ and $D_{F''} = (4,0;((3,4),2),(1,2))$. The weak conjugacy classes of $(\langle \F, \G_1 \rangle, (\G_1,\F))$ and $(\langle \F, \G_2 \rangle, (\G_2,\F))$ are encoded by split metacyclic data sets nos. 3 and 6, respectively, in Table~\ref{tab:split_meta_dsets}.}
	\label{fig:D8_S3_1}
	\end{figure}
	
	 \begin{figure}[htbp]
	\centering
	\labellist
	\tiny
	\pinlabel $(1,3)$ at 115 11
	\pinlabel $(1,3)$ at 115 225
	\pinlabel $\G$ at 218 195
	\pinlabel $\frac{\pi}{2}$ at 210 177
	\pinlabel $(2,3)$ at -15 105
	\pinlabel $(2,3)$ at 240 100
	\pinlabel $(1,3)$ at 180 200
	\pinlabel $(1,3)$ at 50 205
	\pinlabel $(1,3)$ at 180 10
	\pinlabel $(1,3)$ at 50 10
	\pinlabel $(2,3)$ at 5 55
	\pinlabel $(2,3)$ at 220 50
	\pinlabel $(2,3)$ at 5 155
	\pinlabel $(2,3)$ at 220 150
	\endlabellist
	\includegraphics[width=32ex]{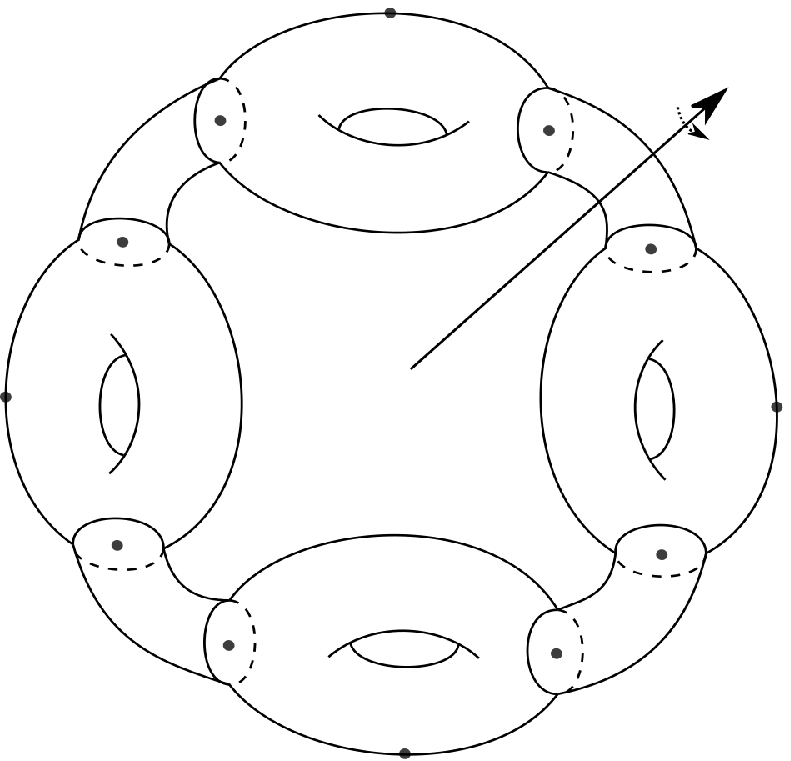}
	\caption{\small A realization of a $\mathbb{Z}_3 \rtimes_{-1} \mathbb{Z}_4$-action $\langle \F, \G \rangle$ on $S_5$, where $D_G = (4,2,1;)$ and $D_F = (3,1;((1,3),2),((2,3),2))$. The action $\F$ is realized via two 1-compatibilities  between the action $\F'$ on two copies of $S_2$ with $D_{F'} = (3,0;((1,3),2),((2,3),2))$. Furthermore, the action $\F'$ is realized by a 1-compatibility between the action $\F''$ and $\F'''$ on $S_1$, where $D_{F''} = (3,0;((1,3),3))$ and $D_{F'''} = (3,0;((2,3),3))$. The weak conjugacy class of $(\langle \F, \G\rangle, (\G,\F))$ is encoded by the split metacyclic data set no. 13 in Table~\ref{tab:split_meta_dsets_s5}.}
	\label{fig:Z3Z4_S5}
	\end{figure}

	 \begin{figure}[H]
	\centering
	\labellist
	\tiny
	\pinlabel $\G_1$ at -8 110
	\pinlabel $\pi$ at 15 124
	\pinlabel $\G_2$ at 128 242
	\pinlabel $\pi$ at 141 221
	\pinlabel $\G_3$ at 252 148
	\pinlabel $\pi$ at 233 125
	\pinlabel $(1,2)$ at 155 185
	\pinlabel $(1,2)$ at 88 190
	\pinlabel $(1,2)$ at 145 140
	\pinlabel $(1,2)$ at 94 140
	\pinlabel $(1,2)$ at 148 70
	\pinlabel $(1,2)$ at 95 70
	\pinlabel $(1,2)$ at 135 8
	\pinlabel $(1,2)$ at 98 10
	\pinlabel $(3,8)$ at 189 148
	\pinlabel $(1,8)$ at 51 148
	\pinlabel $(5,8)$ at 190 62
	\pinlabel $(7,8)$ at 51 62
	\endlabellist
	\includegraphics[width=32ex]{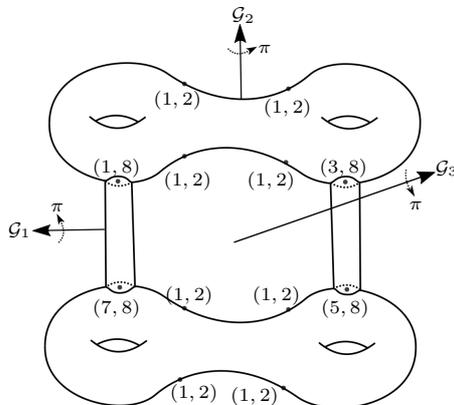}
	\caption{Realization of $\mathbb{Z}_8 \rtimes_{-1} \mathbb{Z}_2$-action $\langle \F, \G_1 \rangle$, $\mathbb{Z}_8 \rtimes_{3} \mathbb{Z}_2$-action $\langle \F, \G_2 \rangle$ and $\mathbb{Z}_8 \rtimes_{5} \mathbb{Z}_2$-action $\langle \F, \G_3 \rangle$ on $S_5$, where $D_{G_1} = D_{G_2} = (2,2;((1,2),4))$, $D_{G_3} = (2,3,1;)$ and $D_F = (8,1;((1,2),2))$. The action $\F$ is realized via two 1-compatibilities  between two actions $\F'$ and $\F''$ on $S_2$ where $D_{F'} = (8,0;(1,2),(1,8),(3,8))$ and $D_{F''} = (8,0;(1,2),(5,8),(7,8))$. The weak conjugacy class of $(\langle \F, \G_i \rangle, (\G_i,\F))$ $1 \leq i \leq 3$ is encoded by the split metacyclic data set nos. 25, 24, and 21, respectively, in Table~\ref{tab:split_meta_dsets_s5}.}
	\label{fig:Z8_Z2_S5}
	\end{figure}

\section{Classification of the weak conjugacy classes in \\ $\Mod(S_3)$ and $\Mod(S_5)$}
\label{sec:classify}

In this section, we will use Theorem~\ref{thm:main} to classify the weak conjugacy classes in $\Mod(S_3)$ and $\Mod(S_5)$. For brevity, we will further assume the following equivalence of the split metacyclic data sets (i.e. the weak conjugacy classes). 
\begin{defn}
\label{defn:eq_data_sets}
Two split metacyclic data sets 
\begin{gather*}
\D=((m \cdot n,k),g_0;[(c_{11},n_{11}),(c_{12},n_{12}),n_1], \cdots , [(c_{\ell 1},n_{\ell 1}),(c_{\ell 2},n_{\ell 2}),n_{\ell} ]) \\ \text{ and } 
\\ \D'=((m \cdot n,k),g_0;[(c'_{11},n'_{11}),(c'_{12},n'_{12}),n'_1], \cdots , [(c'_{\ell 1},n'_{\ell 1}),(c'_{\ell 2},n'_{\ell 2}),n'_{\ell} ])
\end{gather*} 
\noindent are said to be \textit{equivalent} if for each tuple $[(c'_{i1},n'_{i1}),(c'_{i2},n'_{i2}),n'_i]$, there exists a unique tuple $[(c_{j1},n_{j1}),(c_{j2},n_{j2}),n_j]$ satisfying the following conditions:
\begin{enumerate}[(i)]
\item $(c'_{i1},n'_{i1}) = (c_{j1},n_{j1})$,
\item $n'_i = n_j,$ and
\item $c'_{i2} \frac{n}{n'_{i2}} \equiv c_{j2} \frac{n}{n_{j2}} k^{a_i} + b_i (k^{c_{j1} \frac{m}{n_{j1}}} -1) \pmod{n}$ for some $a_i,b_i \in \mathbb{Z}.$
\end{enumerate}
\end{defn}
\noindent Note that equivalent data sets $\D$ and $\D'$ as in Definition~\ref{defn:eq_data_sets} satisfy 
$\D_i' = \D_i$, for $i = 1,2$. We will now provide a classification of the weak conjugacy classes of finite split metacyclic subgroups of $\Mod(S_3)$ and $\Mod(S_5)$ (up to this equivalence) in Tables~\ref{tab:split_meta_dsets} and~\ref{tab:split_meta_dsets_s5}, respectively, 

\begin{landscape}
\begin{table}[htbp]
\small
\begin{center}
    \begin{tabular}{|c|c|c|}
    \hline
      \textbf{Group} & \textbf{Weak conjugacy classes in $\Mod(S_3)$} & \textbf{Cyclic factors $[D_G;D_F]$}\\ \hline
      \multirow{2}{*}{$\Z_3 \rtimes_{-1} \Z_2$} & $((2 \cdot 3,-1),1;[(0,1),(1,3),3])$ & $[(2,2,1;);(3,1;(1,3),(2,3))]$\\ \cline{2-3} 
      & \makecell{$((2 \cdot 3,-1),0;[(1,2),(0,1),2]_3,[(1,2),(1,3),2], [(0,1),(2,3),3])^*$} & $[(2,1;((1,2),4));(3,1;(1,3),(2,3))]$\\
      \hline
       \multirow{4}{*}{$\Z_4 \rtimes_{-1} \Z_2$} & $((2 \cdot 4,-1),1;[(0,1),(1,2),2])$ & $[(2,2,1;);(4,1;((1,2),2))]$\\ \cline{2-3}
        & \makecell{$((2 \cdot 4,-1),0;[(1,2),(0,1),2]_2, [(0,1),(1,4),4], [(0,1),(3,4),4])$} & $[(2,1;((1,2),4));(4,0;((1,4),2),((3,4),2))]$\\ \cline{2-3}
        & \makecell{$((2 \cdot 4,-1),0;[(1,2),(1,4),2]_2, [(0,1),(1,4),4], [(0,1),(3,4),4])$} & $[(2,2,1;);(4,0;((1,4),2),((3,4),2))]$\\ \cline{2-3}
        & \makecell{$((2 \cdot 4,-1),0;[(1,2),(0,1),2]_2,[(1,2),(1,4),2], [(1,2),(3,4),2], [(0,1),(1,2),2])$} & $[(2,1;((1,2),4));(4,1;((1,2),2))]$\\
      \hline
      \multirow{2}{*}{$\Z_3 \rtimes_{-1} \Z_4$} & $((4 \cdot 3,-1),0;[(1,4),(0,1),4],[(1,4),(1,3),4], [(1,2),(2,3),6])$ & $[(4,0;((1,4),2),((1,2),3));(3,1;(1,3),(2,3))]$\\ \cline{2-3}
      & \makecell{$((4 \cdot 3,-1),0;[(3,4),(0,1),4],[(3,4),(1,3),4], [(1,2),(2,3),6])$} & $[(4,0;((3,4),2),((1,2),3));(3,1;(1,3),(2,3))]$\\
      \hline
      \multirow{2}{*}{$\Z_6 \rtimes_{-1} \Z_2$} & $((2 \cdot 6,-1),0;[(1,2),(0,1),2],[(1,2),(1,3),2], [(0,1),(1,2),2], [(0,1),(1,6),6])$ & $[(2,1;((1,2),4);(6,0;((1,2),2),(1,6),(5,6))]$\\ \cline{2-3}
      & \makecell{$((2 \cdot 6,-1),0;[(1,2),(1,6),2],[(1,2),(1,2),2], [(0,1),(1,2),2], [(0,1),(1,6),6])$} & $[(2,2,1;);(6,0;((1,2),2),(1,6),(5,6))]$\\
      \hline
       \multirow{4}{*}{$\Z_4 \rtimes_{-1} \Z_4$} & $((4 \cdot 4,-1),0;[(1,4),(0,1),4],[(0,1),(1,4),4], [(3,4),(1,4),4])^\dagger$ &  $[(4,0;((1,4),2),((1,2),3));(4,0;((1,4),2),((3,4),2))]$ \\ \cline{2-3}
       & \makecell{$((4 \cdot 4,-1),0;[(1,4),(0,1),4],[(1,4),(1,4),4], [(1,2),(3,4),4])^\dagger$} & $[(4,0;((1,4),2),((1,2),3));(4,1;((1,2),2))]$ \\ \cline{2-3}
       & \makecell{$((4 \cdot 4,-1),0;[(3,4),(0,1),4],[(0,1),(1,4),4], [(1,4),(1,4),4])^\dagger$} & $[(4,0;((3,4),2),((1,2),3));(4,0;((1,4),2),((3,4),2))]$ \\ \cline{2-3}
       & \makecell{$((4 \cdot 4,-1),0;[(3,4),(0,1),4],[(3,4),(1,4),4],[(1,2),(3,4),4])^\dagger$} & $[(4,0;((3,4),2),((1,2),3));(4,1;((1,2),2))]$ \\
      \hline
        \multirow{2}{*}{$\Z_8 \rtimes_{5} \Z_2$} & $((2 \cdot 8,5),0;[(1,2),(0,1),2],[(1,2),(7,8),8], [(0,1),(1,8),8])$ & $[(2,1;((1,2),4));(8,0;(1,4),(1,8),(5,8))]$\\ \cline{2-3}
        & \makecell{$((2 \cdot 8,5),0;[(1,2),(0,1),2],[(1,2),(1,8),8], [(0,1),(7,8),8])$} & $[(2,1;((1,2),4));(8,0;(3,4),(3,8),(7,8))]$\\
      \hline
      \multirow{2}{*}{$\Z_7 \rtimes_{2} \Z_3$} & $((3 \cdot 7,2),0;[(1,3),(0,1),3],[(2,3),(6,7),3], [(0,1),(1,7),7])$ & $[(3,1;(1,3),(2,3));(7,0;(1,7),(2,7),(4,7))]$\\ \cline{2-3}
      & \makecell{$((3 \cdot 7,2),0;[(1,3),(0,1),3],[(2,3),(1,7),3], [(0,1),(6,7),7])$} & $[(3,1;(1,3),(2,3));(7,0;(3,7),(6,7),(5,7))]$\\
      \hline
       \multirow{4}{*}{$\Z_{12} \rtimes_{2} \Z_2$} & $((2 \cdot 12,2),0;[(1,2),(0,1),2],[(1,2),(11,12),4], [(0,1),(1,12),12])$ &  $[(2,1;((1,2),4));(12,0;(1,12),(5,12),(1,2))]$ \\ \cline{2-3}
      & \makecell{$((2 \cdot 12,2),0;[(1,2),(0,1),2],[(1,2),(5,12),4], [(0,1),(7,12),12])$} & $[(2,1;((1,2),4));(12,0;(7,12),(11,12),(1,2))]$ \\ \cline{2-3}
      & \makecell{$((2 \cdot 12,2),0;[(1,2),(1,6),2],[(1,2),(1,12),4], [(0,1),(1,12),12])$} & $[(2,2,1;);(12,0;(1,12),(5,12),(1,2))]$ \\ \cline{2-3}
      & \makecell{$((2 \cdot 12,2),0;[(1,2),(1,6),2],[(1,2),(7,12),4], [(0,1),(7,12),12])$} & $[(2,2,1;);(12,0;(7,12),(11,12),(1,2))]$ \\
      \hline
    \end{tabular}
    \caption{The weak conjugacy classes of finite non-abelian split metacyclic subgroups of $\Mod(S_3)$. Note that each data set of type $\dagger$ is quaternionic, and therefore corresponds to the weak conjugacy action of a $Q_{8}$-action on $S_2$.(*The suffix refers to the multiplicity of the tuple in the split metacyclic data set.)}
    \label{tab:split_meta_dsets}
    \end{center}
\end{table}
\end{landscape}

\begin{landscape}
\begin{table}[htbp]
\small
\centering
\resizebox{21cm}{!}
{\begin{tabular}{|c|c|c|}
    \hline
      \textbf{Group} & \textbf{Weak conjugacy classes in $\Mod(S_5)$} & \textbf{Cyclic factors $[D_G;D_F]$}\\ \hline
      \multirow{2}{*}{$\Z_3 \rtimes_{-1} \Z_2$} & $((2 \cdot 3,-1),1;[(0,1),(1,3),3]_2)^*$ & $[(2,3,1;);(3,1;((1,3),2),((2,3),2))]$\\ \cline{2-3} 
      & \makecell{$((2 \cdot 3,-1),0;[(1,2),(0,1),2]_4,[(0,1),(1,3),3]_2)$} & $[(2,2;((1,2),4));(3,1;((1,3),2),((2,3),2))]$\\
      \hline
       \multirow{6}{*}{$\Z_4 \rtimes_{-1} \Z_2$} & $((2 \cdot 4,-1),1;[(1,2),(0,1),2]_2)$ & $[(2,2;((1,2),4));(4,2,1;)]$\\ \cline{2-3}
       & \makecell{$((2 \cdot 4,-1),1;[(1,2),(1,4),2]_2)$} & $[(2,3,1;);(4,2,1;)]$\\ \cline{2-3}
        & \makecell{$((2 \cdot 4,-1),0;[(1,2),(0,1),2]_2,[(0,1),(1,2),2], [(0,1),(1,4),4]_2)$} & $[(2,2;((1,2),4));(4,0;((1,2),2),((1,4),2),((3,4),2))]$\\ \cline{2-3}
        & \makecell{$((2 \cdot 4,-1),0;[(1,2),(1,4),2]_2,[(0,1),(1,2),2], [(0,1),(1,4),4]_2)$} & $[(2,3,1;);(4,0;((1,2),2),((1,4),2),((3,4),2))]$\\ \cline{2-3}
        & \makecell{$((2 \cdot 4,-1),0;[(1,2),(0,1),2]_4,[(1,2),(1,4),2]_2)$} & $[(2,1;((1,2),8));(4,2,1;)]$\\ \cline{2-3}
        & \makecell{$((2 \cdot 4,-1),0;[(1,2),(0,1),2]_2,[(1,2),(1,4),2]_4)$} & $[(2,2;((1,2),4));(4,2,1;)]$\\
      \hline
      \multirow{4}{*}{$\Z_5 \rtimes_{-1} \Z_2$} & $((2 \cdot 5,-1),1;[(0,1),(1,5),5])$ & $[(2,3,1;);(5,1;(1,5),(4,5))]$\\ \cline{2-3}
      & \makecell{$((2 \cdot 5,-1),1;[(0,1),(2,5),5])$} & $[(2,3,1;);(5,1;(2,5),(3,5))]$\\  \cline{2-3}
      & \makecell{$((2 \cdot 5,-1),0;[(1,2),(0,1),2]_3,[(1,2),(4,5),2], [(0,1),(1,5),5])$} & $[(2,2;((1,2),4));(5,1;(1,5),(4,5))]$\\  \cline{2-3}
      & \makecell{$((2 \cdot 5,-1),0;[(1,2),(0,1),2]_3,[(1,2),(3,5),2], [(0,1),(2,5),5])$} & $[(2,2;((1,2),4));(5,1;(2,5),(3,5))]$\\ 
      \hline
      \multirow{3}{*}{$\Z_3 \rtimes_{-1} \Z_4$} & $((4 \cdot 3,-1),1;[(0,1),(1,3),3])$ & $[(4,2,1;);(3,1;((1,3),2),((2,3),2))]$\\ \cline{2-3}
      & \makecell{$((4 \cdot 3,-1),0;[(1,2),(0,1),2],[(0,1),(1,3),3],[(1,4),(0,1),4],[(1,4),(2,3),4])$} & $[(4,0;((1,4),2),((1,2),5));(3,1;((1,3),2),((2,3),2))]$\\ \cline{2-3}
      & \makecell{$((4 \cdot 3,-1),0;[(1,2),(0,1),2],[(0,1),(1,3),3],[(3,4),(0,1),4],[(3,4),(2,3),4])$} & $[(4,0;((3,4),2),((1,2),5));(3,1;((1,3),2),((2,3),2))]$\\
      \hline
      \multirow{4}{*}{$\Z_6 \rtimes_{-1} \Z_2$} & $((2 \cdot 6,-1),1;[(0,1),(1,3),3])$ & $[(2,3,1;);(6,1;(1,3),(2,3))]$\\ \cline{2-3}
      & \makecell{$((2 \cdot 6,-1),0;[(1,2),(0,1),2]_2,[(0,1),(1,6),6],[(0,1),(5,6),6])$} & $[(2,2;((1,2),4));(6,0;((1,6),2),((5,6),2))]$\\ \cline{2-3}
      & \makecell{$((2 \cdot 6,-1),0;[(1,2),(1,6),2]_2,[(0,1),(1,6),6],[(0,1),(5,6),6])$} & $[(2,3,1;);(6,0;((1,6),2),((5,6),2))]$\\ \cline{2-3}
      & \makecell{$((2 \cdot 6,-1),0;[(1,2),(0,1),2],[(1,2),(2,3),2],[(1,2),(1,6),2]_2, [(0,1),(1,3),3])$} & $[(2,2;((1,2),4));(6,1;(1,3),(2,3))]$\\
      \hline
       $\Z_4 \rtimes_{-1} \Z_4$ & $((4 \cdot 4,-1),1;[(0,1),(1,2),2])$ &  $[(4,2,1;);(4,1;((1,2),4))]$ \\ 
      \hline
        \multirow{3}{*}{$\Z_8 \rtimes_{5} \Z_2$} & $((2 \cdot 8,5),1;[(0,1),(1,2),2])$ & $[(2,3,1;);(8,1;((1,2),2))]$\\ \cline{2-3}
        & \makecell{$((2 \cdot 8,5),0;[(1,2),(1,4),4],[(0,1),(1,8),8],[(1,2),(1,8),8])$} & $[(2,3,1;);(8,0;(1,2),(3,4),(1,8),(5,8))]$\\  \cline{2-3}
        & \makecell{$((2 \cdot 8,5),0;[(1,2),(1,4),4],[(0,1),(3,8),8],[(1,2),(3,8),8])$} & $[(2,3,1;);(8,0;(1,2),(1,4),(3,8),(7,8))]$\\
      \hline
      $\Z_8 \rtimes_{3} \Z_2$ & $((2 \cdot 8,3),0;[(1,2),(0,1),2],[(1,2),(1,4),2],[(1,2),(1,8),4],[(1,2),(3,8),4])$ &  $[(2,2;((1,2),4));(8,1;((1,2),2))]$ \\
      \hline
      $\Z_8 \rtimes_{-1} \Z_2$ & $((2 \cdot 8,-1),0;[(1,2),(0,1),2]_2,[(1,2),(1,8),2],[(1,2),(5,8),2],[(0,1),(1,2),2])$ &  $[(2,2;((1,2),4));(8,1;((1,2),2))]$ \\
      \hline
      \multirow{4}{*}{$\Z_5 \rtimes_{-1} \Z_4$} & $((4 \cdot 5,-1),0;[(1,4),(0,1),4],[(1,4),(1,5),4],[(1,2),(4,5),10])$ & $[(4,0;((1,4),2),((1,2),5));(5,1;(2,5),(3,5))]$\\ \cline{2-3}
      & \makecell{$((4 \cdot 5,-1),0;[(1,4),(0,1),4],[(1,4),(2,5),4],[(1,2),(3,5),10])$} & $[(4,0;((1,4),2),((1,2),5));(5,1;(1,5),(4,5))]$\\  \cline{2-3}
      & \makecell{$((4 \cdot 5,-1),0;[(3,4),(0,1),4],[(3,4),(1,5),4],[(1,2),(4,5),10])$} & $[(4,0;((3,4),2),((1,2),5));(5,1;(2,5),(3,5))]$\\ \cline{2-3}
      & \makecell{$((4 \cdot 5,-1),0;[(3,4),(0,1),4],[(3,4),(2,5),4],[(1,2),(3,5),10])$} & $[(4,0;((3,4),2),((1,2),5));(5,1;(1,5),(4,5))]$\\
      \hline
      \multirow{4}{*}{$\Z_{10} \rtimes_{-1} \Z_2$} & $((2 \cdot 10,-1),0;[(1,2),(0,1),2],[(1,2),(1,5),2],[(0,1),(1,2),2],[(0,1),(3,10),10])$ & $[(2,2;((1,2),4));(10,0;((1,2),2),(3,10),(7,10))]$\\ \cline{2-3}
      & \makecell{$((2 \cdot 10,-1),0;[(1,2),(0,1),2],[(1,2),(2,5),2],[(0,1),(1,2),2],[(0,1),(1,10),10])$} & $[(2,2;((1,2),4));(10,0;((1,2),2),(1,10),(9,10))]$\\ \cline{2-3}
      & \makecell{$((2 \cdot 10,-1),0;[(1,2),(1,10),2],[(1,2),(3,10),2],[(0,1),(1,2),2],[(0,1),(3,10),10])$} & $[(2,3,1;);(10,0;((1,2),2),(3,10),(7,10))]$\\ \cline{2-3}
      & \makecell{$((2 \cdot 10,-1),0;[(1,2),(1,10),2],[(1,2),(1,2),2],[(0,1),(1,2),2],[(0,1),(1,10),10])$} & $[(2,3,1;);(10,0;((1,2),2),(1,10),(9,10))]$\\
      \hline
    \end{tabular}}
    \caption{The weak conjugacy classes of finite non-abelian split metacyclic subgroups of $\Mod(S_5)$.(*The suffix refers to the multiplicity of the tuple in the split metacyclic data set.)}
    \label{tab:split_meta_dsets_s5}
\end{table}
\end{landscape}

\begin{landscape}
	\centering
	Continuation of Table 2. \vspace{5mm}
	\begin{table*}[htbp]
		\small
		\centering
		\resizebox{21cm}{!}
		{\begin{tabular}{|c|c|c|}
				\hline
				\textbf{Group} & \textbf{Weak conjugacy classes in $\Mod(S_5)$} & \textbf{Cyclic factors $[D_G;D_F]$}\\ \hline
				\multirow{4}{*}{$\Z_6 \rtimes_{-1} \Z_4$} & $((4 \cdot 6,-1),0;[(1,4),(0,1),4],[(3,4),(1,6),4],[(0,1),(5,6),6])$ & $[(4,0;((1,4),2),((1,2),5));(6,0;((1,6),2),((5,6),2))]$\\ \cline{2-3}
				& \makecell{$((4 \cdot 6,-1),0;[(3,4),(0,1),4],[(1,4),(1,6),4],[(0,1),(5,6),6])$} & $[(4,0;((3,4),2),((1,2),5));(6,0;((1,6),2),((5,6),2))]$\\ \cline{2-3}
				& \makecell{$((4 \cdot 6,-1),0;[(1,4),(0,1),4],[(1,4),(1,6),4],[(1,2),(5,6),6])$} & $[(4,0;((1,4),2),((1,2),5));(6,1;(1,3),(2,3))]$\\ \cline{2-3}
				& \makecell{$((4 \cdot 6,-1),0;[(3,4),(0,1),4],[(3,4),(1,6),4],[(1,2),(5,6),6])$} & $[(4,0;((3,4),2),((1,2),5));(6,1;(1,3),(2,3))]$\\
				\hline
				\multirow{4}{*}{$\Z_{15} \rtimes_{4} \Z_2$} & $((2 \cdot 15,4),0;[(1,2),(0,1),2],[(1,2),(14,15),6],[(0,1),(1,15),15])$ & $[(2,2;((1,2),4));(15,0;(1,15),(4,15),(2,3))]$\\ \cline{2-3}
				& \makecell{$((2 \cdot 15,4),0;[(1,2),(0,1),2],[(1,2),(13,15),6],[(0,1),(2,15),15])$} & $[(2,2;((1,2),4));(15,0;(2,15),(8,15),(1,3))]$\\  \cline{2-3}
				& \makecell{$((2 \cdot 15,4),0;[(1,2),(0,1),2],[(1,2),(8,15),6],[(0,1),(7,15),15])$} & $[(2,2;((1,2),4));(15,0;(7,15),(13,15),(2,3))]$\\  \cline{2-3}
				& \makecell{$((2 \cdot 15,4),0;[(1,2),(0,1),2],[(1,2),(4,15),6],[(0,1),(11,15),15])$} & $[(2,2;((1,2),4));(15,0;(11,15),(14,15),(1,3))]$\\ 
				\hline
				\multirow{8}{*}{$\Z_{20} \rtimes_{9} \Z_2$} & $((2 \cdot 20,9),0;[(1,2),(0,1),2],[(1,2),(19,20),4],[(0,1),(1,20),20])$ & $[(2,2;((1,2),4));(20,0;(1,20),(9,20),(1,2))]$\\ \cline{2-3}
				& \makecell{$((2 \cdot 20,9),0;[(1,2),(1,10),2],[(1,2),(1,20),4],[(0,1),(1,20),20])$} & $[(2,3,1;);(20,0;(1,20),(9,20),(1,2))]$\\ \cline{2-3}
				& \makecell{$((2 \cdot 20,9),0;[(1,2),(0,1),2],[(1,2),(17,20),4],[(0,1),(3,20),20])$} & $[(2,2;((1,2),4));(20,0;(3,20),(7,20),(1,2))]$\\ \cline{2-3}
				& \makecell{$((2 \cdot 20,9),0;[(1,2),(1,10),2],[(1,2),(19,20),4],[(0,1),(3,20),20])$} & $[(2,3,1;);(20,0;(3,20),(7,20),(1,2))]$\\ \cline{2-3}
				& \makecell{$((2 \cdot 20,9),0;[(1,2),(0,1),2],[(1,2),(9,20),4],[(0,1),(11,20),20])$} & $[(2,2;((1,2),4));(20,0;(11,20),(19,20),(1,2))]$\\ \cline{2-3}
				& \makecell{$((2 \cdot 20,9),0;[(1,2),(1,10),2],[(1,2),(11,20),4],[(0,1),(11,20),20])$} & $[(2,3,1;);(20,0;(11,20),(19,20),(1,2))]$\\ \cline{2-3}
				& \makecell{$((2 \cdot 20,9),0;[(1,2),(0,1),2],[(1,2),(7,20),4],[(0,1),(13,20),20])$} & $[(2,2;((1,2),4));(20,0;(13,20),(17,20),(1,2))]$\\ \cline{2-3}
				& \makecell{$((2 \cdot 20,9),0;[(1,2),(1,10),2],[(1,2),(9,20),4],[(0,1),(13,20),20])$} & $[(2,3,1;);(20,0;(13,20),(17,20),(1,2))]$\\
				\hline
				
		\end{tabular}}

	\end{table*}
\end{landscape}

\section*{Acknowledgements}
The first and third authors were supported by the UGC-JRF fellowship. The authors would also like to thank Dr. Siddhartha Sarkar for some helpful discussions.

\bibliographystyle{plain} 
\bibliography{meta_real}
\end{document}